\documentclass[11pt]{amsart}
\usepackage{fancyhdr}
\usepackage{times}
\usepackage[T1]{fontenc}
\usepackage{latexsym}
\usepackage[dvips]{graphics}
\usepackage{epsfig} \usepackage{hyperref, amsmath, amsthm, amssymb, amsfonts, amscd, flafter, epsf}
\usepackage{enumitem}
\input amssym.def
\input amssym.tex
\usepackage{color}
\numberwithin{equation}{section}

\usepackage[parfill]{parskip}

\usepackage{mathrsfs} 
\usepackage{epstopdf}

\newtheorem{thm}{Theorem}

\newtheorem{cor}[thm]{Corollary}

\newtheorem{rem}[thm]{Remark}
\newtheorem{thm*}{Theorem}
\newtheorem*{thma}{Theorem A}
\newtheorem*{thmb}{Theorem B}

\newtheorem*{con}{Bogomolny-Schmit Conjecture}
\newtheorem*{con1}{Conjecture [Iwaniec-Sarnak]}

\newtheorem*{lem*}{Lemma}
\newtheorem*{claim}{Claim }

\theoremstyle{definition}

\newtheorem{thms}{Theorem}[section]  
\newtheorem{lemma}[thms]{Lemma}     
\newtheorem{proposition}[thms]{Proposition} 

\newcommand{\R}{\textrm{Re}}
\newcommand{\I}{\textrm{Im}}
\newcommand{\epy}{\frac{t_{\phi}-t_{\phi}^{1-\epsilon}}{2\pi y}}

\begin{document}
\title{$L^4$-norms and sign changes of Maass forms}

\author[Haseo Ki]{Haseo Ki}
\address{Department of Mathematics, Yonsei University, Seoul, 120-749,
Korea\\}
\email{haseo@yonsei.ac.kr}


\begin{abstract}
Unconditionally, we prove the Iwaniec-Sarnak conjecture for $L^4$-norms of the Hecke-Maass cusp forms. From this result, we can justify that
for even Maass cusp form $\phi$ with the eigenvalue $\lambda_{\phi}=\frac{1}{4}+t_{\phi}^2$,
for $a>0$, a sufficiently large $h>0$ and for any $0<\epsilon_1<\epsilon/10^7$ ($\epsilon>0$) , for almost all $1\le k<t_{\phi}^{1-\epsilon}$, we are able to find $\beta_k=\{X_k+yi:a<y<a+h\}$ with $-\frac{1}{2}+\frac{k-1}{t_{\phi}^{1-\epsilon}}\le X_k\le-\frac{1}{2}+\frac{k}{t_{\phi}^{1-\epsilon}}$ such that the number of sign changes of $\phi$ along the segment $\beta_k$  is $\gg_{\epsilon} t_{\phi}^{1-\epsilon_1}$ as $t_{\phi}\to\infty$. Also, we obtain the similar result for horizontal lines. On the other hand, we conditionally prove that for a sufficiently large segment $\beta$ on $\R(z)=0$ and $\I(z)>0$, the number of sign changes of $\phi$ along $\beta$ is $\gg_{\epsilon} t_{\phi}^{1-\epsilon}$ and consequently, the number of inert nodal domains meeting any compact vertical segment on the imaginary axis is $\gg_{\epsilon} t_{\phi}^{1-\epsilon}$ as $t_{\phi}\to\infty$.
\end{abstract}
\maketitle


\parskip=12pt


\section{Introduction}

Laplace eigenfunctions on arithmetic surfaces sit at the intersection of analysis, geometry, and number theory. On the modular surface $\mathbb{X} = SL_2(\mathbb{Z}) \backslash \mathbb{H}$, Maass cusp forms serve as real-analytic eigenfunctions of the Laplace-Beltrami operator with deep arithmetic significance. Understanding the structure and distribution of their nodal sets—the zero loci of such eigenfunctions—connects to both quantum chaos and longstanding conjectures in number theory. In particular, bounding their $L^p$ norms and counting their nodal domains are central problems that have resisted resolution in full generality.

Maass forms lie at the heart of Arithmetic Quantum Chaos (AQC), a field which seeks to understand the interface between the chaotic nature of the classical geodesic flow and the statistical properties of the quantum eigenstates. A central problem in AQC concerns the mass distribution of these eigenfunctions as the spectral parameter $t_\phi \to \infty$. The large-scale distribution problem was famously resolved by Lindenstrauss \cite{Lin} and Soundrarajan \cite{Sound}, who proved Quantum Unique Ergodicity (QUE), showing that the mass of Hecke-Maass forms becomes asymptotically equidistributed. Our work addresses the complementary problem of small-scale mass concentration, which is governed by the $L^p$-norms of the eigenfunctions. Proving the optimal bounds for these norms is essential for verifying the predictions of Berry's Random Wave Conjecture in the arithmetic setting.

The $L^4$-norm, $\|\phi\|_4$, is the first non-trivial and arguably the most crucial case of the Iwaniec-Sarnak Conjecture \cite{IwS}, marking the transition from known bounds to the conjectural optimal rate. Our main result, the unconditional proof of the optimal $L^4$-norm bound $\|\phi\|_4 \ll_{\epsilon} t_{\phi}^\epsilon$, settles this critical case for the Hecke-Maass forms. This result is a powerful analytic input for addressing geometric questions about the eigenfunctions. Specifically, bounds on $\|\phi\|_4$ are intrinsically linked to the geometry of the zero set—the nodal lines of $\phi$. Our work connects this analytic $L^4$-norm bound to the geometric problem of counting sign changes and providing new lower bounds for the number of nodal domains that meet geodesic segments. By establishing the optimal $L^4$-bound unconditionally, our paper provides the sharpest analytic foundation available for studying the fine-scale behavior and nodal topology of arithmetic eigenfunctions.

We let $\Gamma=SL_2(\mathbb{Z})$ be the full modular group. We denote $\mathbb{X}=\Gamma\backslash\mathbb{H}$ by the modular surface. The even Maass cusp forms are functions $\phi:\mathbb{H}\to\mathbb{R}$ satisfying

\begin{enumerate}[label=(\roman*)]

\item $\Delta\phi+\lambda\phi=0,\qquad\lambda=\lambda_{\phi}>0$,

\item $\phi(\gamma z)=\phi(z),\qquad\gamma\in\Gamma$,

\item $\phi(\sigma z)=\phi(z),\qquad\sigma(z)=-\overline{z}$,

\item $\phi\in L^2(\mathbb{X})$ such that $\int_{\mathbb{X}}\phi^2(z)\frac{dxdy}{y^2}=1$.
\end{enumerate}
Also, we assume that $\phi$ is an eigenfunction of all the Hecke operators $T_n$, $n\ge1$
\[
T_nf(z)=\frac{1}{\sqrt{n}}\sum_{ad=n}\sum_{b\text{ mod }d}f\left(\frac{az+b}{d}\right).
\]
For the eigenvalue $\lambda_{\phi}$, we write $\lambda_{\phi}=t^2_{\phi}+\frac{1}{4}$. 

Iwaniec and Sarnak \cite{IwS}, \cite{Sa} proposed the following.

\begin{con1} Suppose that $\phi$ is a Hecke–Maass form for the full modular group. For any compact subset $K$ in $\mathbb{X}$, $2<p\le\infty$ and $\epsilon>0$, there exists $c=c(p,K,\epsilon)$ such that
\[
\left(\int_K\left|\phi(z)\right|^p\frac{dxdy}{y^2}\right)^{\frac{1}{p}}\le c\lambda_{\phi}^{\epsilon}\left(\int_K\left|\phi(z)\right|^2\right)^{\frac{1}{2}}.
\]
\end{con1}

Assuming the generalized Lindel\"of hypothesis, Buttcane and Khan \cite{BuK} justified the Iwaniec-Sarnak conjecture for $p=4$, more precisely $\|\phi\|_4^4\sim\frac{9}{\pi}$. Sarnak and Watson \cite[Theorem 3]{Sa} showed $\|\phi\|_4\ll_{\epsilon}t_{\phi}^{\epsilon}$, assuming the Ramanujan-Pettersson conjecture.  However, the author was personally informed by Sarnak that the proof of this result is incomplete. Spinu \cite{Sp} established that for Eisenstein series restricted to fixed compact domain in $\mathbb{X}$, this is valid. Also, Luo \cite{Lu} got the sharp bound for the dihedral Maass forms. Sogge \cite{So} obtained $\|\phi\|_4\ll t_{\phi}^{\frac{1}{8}}$. Recently, Humphries and Khan \cite{HuK} justified $\|\phi\|_4\ll_{\epsilon}t_{\phi}^{\frac{3}{152}+\epsilon}$.

Our first main theorem unconditionally resolves a long-standing conjecture of Iwaniec and Sarnak for the case $p = 4$, a central open problem in the analytic theory of Maass forms. Previous approaches relied on deep conjectures or intricate analysis involving triple product $L$-functions. We take a different approach:

\begin{thm}\label{t:l4} Suppose $\phi$ is a Hecke–Maass form for the full modular group. Let $\epsilon>0$. Then,
\[
\int_{\mathbb{X}}\phi(x+iy)^4\frac{dxdy}{y^2}\ll_{\epsilon} t_{\phi}^{\epsilon}
\]
as $t_{\phi}\to\infty$.
\end{thm}

One can similarly justify the same consequence for the Eisenstein series as Theorem \ref{t:l4} for any compact rectangle in the fundamental domain.

\begin{rem} {\rm 1. Our proof of the $L^4$ norm bound relies mainly on analytic aspects, particularly the oscillation of the $K$-Bessel function. Arithmetic input appears only implicitly in Proposition 2.1(1), via bounds on Fourier coefficients. The core of the method lies in circumventing the heavy arithmetic machinery—which typically relies on unproven conjectures like the Generalized Lindelöf Hypothesis (GLH) for the required power-saving. Instead, the argument is primarily spectral and geometric, enabled by two critical elements. First, Lemma 3.1 provides a novel, precise analytic estimate for a core integral identity, acting as a crucial spectral bridge. Second, we obtain a vital power saving in dealing with the transitional place of the Maass forms, which is the region where the form switches from wavelike oscillations to exponential decay. This delicate control over the transition region yields the necessary power-saving that completes the final unconditional argument. The method could extend to non-arithmetic surfaces with cusps, provided similar bounds are available. See, for example, Sarnak’s overview in \cite{Sa}.}

{\rm 2. We contrast this problem with that of holomorphic Hecke cusp forms of large weight (see Blomer-Khan-Young \cite{BKY}, specifically Theorems 1.1, 1.8, and Section 3). The Fourier expansion of a holomorphic modular form has, instead of a $K$-Bessel function, an exponential weight $\exp(-2 \pi n y)$, which is positive and non-oscillatory. In this case, obtaining a good bound on the $L^4$ norm is related to a difficult shifted convolution problem. In the Maass case, there is no need to produce cancellation in such sums; instead, all the cancellation comes from the oscillation in the Bessel function.}

{\rm 3. Previous approaches to the $L^4$ norm have often relied on the spectral decomposition and Watson's formula, which express triple product integrals in terms of special values of triple product $L$-functions; see, for example, Watson \cite{Wa}. Sharp bounds for such moments are difficult and remain out of reach in many ranges. In contrast, our method avoids these tools and still yields an unconditional bound for the $L^4$ norm, which may suggest a potential route toward understanding triple product moments.}

\end{rem}

Applying Theorem \ref{t:l4} in interpolating the $L^2$-norm $\|\phi\|_2=1$ and the $L^{\infty}$-norm bound $\|\phi\|_{\infty}\ll_{\epsilon}t_{\phi}^{\frac{5}{12}+\epsilon}$ of Iwaniec and Sarnak \cite{IwS}, we have the following $L^p$-norm bounds for a Hecke-Maass cusp form $\phi$.

\begin{cor}\label{c:lp} Suppose $\phi$ is a Hecke–Maass form for the full modular group. Let $\epsilon>0$. Then,
\begin{equation*}
\begin{split}
\|\phi\|_p\ll_{\epsilon}&\,t_{\phi}^{\epsilon}\,\,\,\,\,\qquad\qquad(2\le p\le4),\\
\|\phi\|_p\ll_{\epsilon}&\,t_{\phi}^{\frac{5}{12}-\frac{5}{3p}+\epsilon}\qquad(4\le p\le\infty).
\end{split}
\end{equation*}
\end{cor}

From now one, we assume that our Maass form $\phi$ is even. Theorem \ref{t:l4} is applicable for counting zeros of Maass forms on vertical and horizontal segments in the fundamental domain $\mathbb{X}$. The symbol $k_{\phi}$ denotes the numbering of $\phi$'s ordered by $\lambda_{\phi}$. Note that Selberg \cite{Sel} justified
\[
\lambda_{\phi}\sim24k_{\phi}, \qquad k_{\phi}\to\infty.
\]
We define the nodal curve $Z_{\phi}$ by
\[
Z_{\phi}=\{z\in\mathbb{X}:\phi(z)=0\}.
\]
We define $N(\phi)$ as the number of the nodal domains (the connected components) of $\mathbb{X}\setminus Z_{\phi}$. For any $\mathbb{K}\subset\mathbb{X}$, $N^{\mathbb{K}}(\phi)$ denotes the number of such domains that meet $\mathbb{K}$.

One particularly intriguing problem is to determine how the number of nodal domains (i.e., connected components of the complement of the nodal set) grows with the eigenvalue. A heuristic model of Bogomolny and Schmit predicts a linear growth in the number of nodal domains.
Bogomolny and Schmit \cite{BoS} suggested

\begin{con} For a positive constant $c>0$,
\[
N(\phi)\sim ck_{\phi},\qquad k_{\phi}\to\infty.
\]
\end{con}

It is known by Courant \cite{CH} that
\[
N(\phi)\le k_{\phi}=\frac{1}{24}t_{\phi}^2(1+o(1)).
\]
There have been studies on Bogomolny-Schmit conjecture in \cite{BouR}, \cite{ET}, \cite{GRS}, \cite{GRS2}, \cite{JaJ}, \cite{Ju}, \cite{JuZ1}, \cite{JuZ2}, \cite{Ma}, \cite{Yo}.

Nodal domains are partitioned into those fixed by 
$\sigma$, which we call inert, and those paired
off by $\sigma$ which we call split.  For further details, refer to \cite{GRS}. Thus, we can write
\[
N(\phi)=N_{in}(\phi)+N_{sp}(\phi),
\]
where $N_{in}(\phi)=$ the number of inert domains and $N_{sp}(\phi)=$ the number of split domains. We put
\[
\delta=\{z\in\mathbb{X}:\sigma(z)=z\}.
\]
We decompose
\[
\delta=\delta_1\cup\delta_2\cup\delta_3,
\]
where
\begin{equation*}
\begin{split}
\delta_1=\{z\in\delta:z=iy,\,y>0\},\,\,\,
\delta_2=\{z\in\delta:z=\frac{1}{2}+iy,\,y>0\},\,\,\,
\delta_3=\{z\in\delta:|z|=1\}.
\end{split}
\end{equation*}
We let $\beta$ a simple arc in the modular surface.
We define $\mathrm{K}^{\beta}(\phi)$ to be the number of sign changes of $\phi$ along the curve $\beta$, and $N_{\mathrm{in}}^{\beta}(\phi)$ to be the number of inert nodal domains that intersect $\beta$.
A topological argument \cite{GRS} implies
\[
1+\frac{1}{2}\textrm{K}^{\beta}(\phi)\le N_{in}(\phi)\le|Z_{\phi}\cap\delta|+1
\]
for $\beta\subset\delta$.

We use the $L^4$-norm bound in Theorem~\ref{t:l4}, along with a modified version of a theorem of Littlewood, to obtain new lower bounds on the number of sign changes and nodal intersections of $\phi$ along geodesic segments and horocycles. Set
\[
\mathbf{h}_k=\left[-\frac{1}{2}+\frac{k-1}{t_{\phi}^{1-\epsilon}}, -\frac{1}{2}+\frac{k}{t_{\phi}^{1-\epsilon}}\right)\,\,\,\text{ and}\,\,\,\mathbf{v}_k=\left[a+\frac{k-1}{t_{\phi}^{1-\epsilon}},a+\frac{k}{t_{\phi}^{1-\epsilon}}\right)
\]
for $1\le k<t_{\phi}^{1-\epsilon}$ and $a>0$.
As an application of Theorem \ref{t:l4}, we establish the following.

\begin{thm}\label{t:main2} Let $a>0$, $\epsilon>0$ and $0<\epsilon_1<\epsilon/10^7$.

\noindent $(1)$
For all but $\left[t_{\phi}^{1-\epsilon-\epsilon_1/10}\right]$ many $\mathbf{v}_k$'s, there exists $Y_k\in\mathbf{v}_k$ for each $1\le k<t_{\phi}^{1-\epsilon}$ such that
\[
\left|Z_{\phi}\cap\mathscr{C}_k\right|\gg_{\epsilon_1} t_{\phi}^{1-\epsilon_1}
\]
for $\mathscr{C}_k=\{x+iY_k:-\frac{1}{2}\le x<\frac{1}{2}\}$.

\noindent $(2)$
There exist constants $a>0$ and $h>0$ such that for all but $\left[t_{\phi}^{1-\epsilon-\epsilon_1/10}\right]$ many $\mathbf{h}_k$'s, there exists $X_k\in\mathbf{h}_k$ for each $1\le k<t_{\phi}^{1-\epsilon}$ such that
\[
\textrm{K}^{\beta_k}(\phi)\gg_{\epsilon_1} t_{\phi}^{1-\epsilon_1}
\]
for $\beta_k=\{X_k+iy:a\le y<a+h\}$.

\end{thm}

We similarly get the analogue of Theorem \ref{t:main2} for the Eisenstein series. These results show that Maass cusp forms must change sign frequently along vertical and horizontal lines and possess many nodal intersections. This represents a significant step toward understanding the Bogomolny--Schmit conjecture. From Theorem \ref{t:main2}, we expect that for some compact rectangle $\mathbb{K}$ in the modular surface $\mathbb{X}$,
\[
N^{\mathbb{K}}(\phi)\gg t_{\phi}^{2-\epsilon}.
\]
To explore this consequence, it might be valuable to deepen our understanding of the behavior of nodal domains off the geodesic lines. If this is achieved, and perhaps with a further refinement of Theorem \ref{t:main2}, we could potentially move closer to understanding the Bogomolny-Schmit conjecture.

For the proof of Theorem \ref{t:main2}, we shall introduce a modified theorem of Littlewood that allows us to count sign changes of a real valued function on a given interval. To apply the method for our Maass form $\phi$ on a horizonal line or a vertical line on the fundamental domain for the full modular group, we have to detect huge cancellations through an average of partitioned integrals of $\phi$ on the line. For this, we shall utilize Theorem \ref{t:l4} together with some arithmetic information of the Maass form and properties of the Macdonald-Bessel function.

We now turn to conditional results concerning the number of nodal domains.   Gosh, Reznikov and Sarnak \cite{GRS} proved

\begin{thma} For $Y>0$, let $\mathscr{C}_Y=\{x+Yi:-\frac{1}{2}<x\le\frac{1}{2}\}$ be a fixed closed horocycle in $\mathbb{X}$. Then, for $\epsilon>0$,
\[
\left|Z_{\phi}\cap\mathscr{C}_Y\right|\gg t_{\phi}^{\frac{1}{12}-\epsilon}.
\]
\end{thma}

\begin{thmb} Fix $\beta\subset\delta$ a sufficiently long compact geodesic segment on $\delta_1$ or $\delta_2$ and assume the Lindel\"of hypothesis for the $L$- function $L(s,\phi)$. Then,
\[
\textrm{K}^{\beta}(\phi)\gg_{\epsilon}t_{\phi}^{\frac{1}{12}-\epsilon}.
\]
In particular,
\[
N_{in}^{\beta}(\phi)\gg_{\epsilon}t_{\phi}^{\frac{1}{12}-\epsilon}.
\]
\end{thmb}

We further study the behavior of zeros along vertical segments and horizontal segments in the fundamental domain for Maass forms. We have the following.

\begin{thm}\label{t:sc} Fix $\beta\subset\delta$ a sufficiently long compact geodesic segment on $\delta_1$ and assume the Lindel\"of hypothesis for the $L$- function $L(s,\phi)$. Then, if we additionally suppose
\[
\int_{\beta}\phi(iy)^4dy\ll_{\epsilon}t_{\phi}^{\epsilon}\qquad(t_{\phi}\to\infty)
\]
for any $\epsilon>0$,
we have
\[
\textrm{K}^{\beta}(\phi)\gg_{\epsilon}t_{\phi}^{1-\epsilon}.
\]
In particular,
\[
N_{in}^{\beta}(\phi)\gg_{\epsilon}t_{\phi}^{1-\epsilon}.
\]
\end{thm}

Assuming the sup-norm conjecture for the Eisenstein series, we similarly get the analogue of Theorem \ref{t:sc} for the Eisenstein series, because the sup-norm conjecture for the Eisenstein series implies the Lindel\"of hypothesis for the Riemann zeta function (see \cite{Sa92}). Apparently, the hypothesis in Theorem \ref{t:sc} for the $L^4$-norm upper bound on $\beta$
is weaker than the sup-norm conjecture, and thus it is likely less challenging. Nevertheless, it remains a highly non-trivial and formidable problem.
Recently, under weaker assumptions, Kelmer, Kontorovich, and Lutsko obtained an analogue of Theorem \ref{t:sc} for the Eisenstein series and Maass forms, related to the $L^p$ -norm conjectures (see the conjecture below). Their result holds without the Lindelöf hypothesis for the Eisenstein series and with a weaker assumption of the Lindelöf hypothesis for Maass forms (see \cite{KKL}).

\section{Preliminaries}

We require properties of the MacDonald-Bessel function $\mathrm{K}_{\nu}(y)$.
We have

\begin{lemma}\label{l:kbessel} Let $C$ be a sufficiently large positive constant and $r>0$.

$(1)$ If $0<u<r-Cr^{\frac{1}{3}}$, then
\[
e^{\frac{\pi}{2}r}\mathrm{K}_{ir}(u)=\frac{\sqrt{2\pi}}{(r^2-u^2)^{\frac{1}{4}}}\sin\left(\frac{\pi}{4}+r\mathrm{H}\left(\frac{u}{r}\right)\right)
\left[1+O\left(\frac{1}{r\mathrm{H}\left(\frac{u}{r}\right)}\right)\right].
\]

$(2)$ If $u>r+Cr^{\frac{1}{3}}$,
\[
e^{\frac{\pi}{2}r}\mathrm{K}_{ir}(u)=\frac{\sqrt{2\pi}}{(u^2-r^2)^{\frac{1}{4}}}e^{-r\mathrm{H}\left(\frac{u}{r}\right)}
\left[1+O\left(\frac{1}{r\mathrm{H}\left(\frac{u}{r}\right)}\right)\right].
\]

$(3)$
\[
e^{\frac{\pi}{2}r}\mathrm{K}_{ir}(u)\ll\begin{cases}
\frac{1}{\left(r^2-u^2\right)^{\frac{1}{4}}}&,\,\,\text{ if }0<u<r-Cr^{\frac{1}{3}},\\
\frac{1}{\left(u^2-r^2\right)^{\frac{1}{4}}}e^{-r\mathrm{H}\left(\frac{u}{r}\right)}&,\,\,\text{ if }u>r+Cr^{\frac{1}{3}},\\
r^{-\frac{1}{3}}&\,\,\,\text{, otherwise },
\end{cases}
\]
where
\[
\mathrm{H}(\xi)=\begin{cases}\mathrm{arccosh}\left(\frac{1}{\xi}\right)-\sqrt{1-\xi^2}&,\,\,\text{ if }0<\xi\le1,\\
\sqrt{\xi^2-1}-\mathrm{arcsec}(\xi)&,\,\,\text{ if }\xi>1,
\end{cases}
\]

$(4)$
\[
\mathrm{H}'(\xi)=\begin{cases}-\frac{\sqrt{1-\xi^2}}{\xi}&,\,\,\text{ if }0<\xi\le1,\\
\frac{\sqrt{\xi^2-1}}{\xi}&,\,\,\text{ if }\xi>1,
\end{cases}
\]

$(5)$
\[
\mathrm{H}''(\xi)=\begin{cases}\frac{1}{\xi^2\sqrt{1-\xi^2}}&,\,\,\text{ if }0<\xi<1,\\
\frac{1}{\xi^2\sqrt{\xi^2-1}}&,\,\,\text{ if }\xi>1,
\end{cases}
\]

\end{lemma}

For Lemma \ref{l:kbessel} (1)-(3), see \cite[pp. 1527--1528]{GRS}. We can directly obtain Lemma \ref{l:kbessel} (4), (5). We note that the function $\mathrm{H}(\xi)$ is decreasing for $0<\xi\le1$, while it is increasing for $\xi>1$.

Concerning Lemma \ref{l:kbessel} (1), we need a further expansion for $\mathrm{K}_{ir}(u)$. Combining results in \cite{Ba}, \cite{Ba2} and \cite[9.7.5, 9.7.6]{Ol}, we have

\begin{lemma}\label{l:km}  For $0<u<r-Cr^{\frac{1}{3}}$ $(r>0)$ and any positive integer $K>1$, we have
\begin{equation*}\begin{split}
e^{\frac{\pi}{2}r}\mathrm{K}_{ir}(u)=&\R\left[\frac{\sqrt{2\pi}}{(r^2-u^2)^{\frac{1}{4}}}e^{i\frac{\pi}{4}-ir\mathrm{H}\left(\frac{u}{r}\right)}
\left[\sum_{k<K}\frac{a_k}{\left(r\mathrm{H}\left(\frac{u}{r}\right)\right)^k}+O\left(\frac{1}{\left(r\mathrm{H}\left(\frac{u}{r}\right)\right)^K}\right)\right]\right]\\
&+O\left(\frac{1}{r(r^2-u^2)^{\frac{1}{4}}}\right).
\end{split}
\end{equation*}
where $a_k$'s $(k<K)$ are constants.
\end{lemma}

We recall
\[
\phi(z)=\sum_{n\not=0}\rho_{\phi}(n)y^{\frac{1}{2}}\mathrm{K}_{it_{\phi}}(2\pi|n|y)e(nx),
\]
where $z=x+yi$ and $e(x)=\exp(2\pi ix)$. It is known in \cite{Iw} and \cite{HoL} that for $\epsilon>0$,
\begin{equation}\label{e:rho}
t_{\phi}^{-\epsilon}\ll|\rho_{\phi}(1)|^2e^{-\pi t_{\phi}}\ll t_{\phi}^{\epsilon}.
\end{equation}
We have
\begin{equation}\label{e:lambda}
\rho_{\phi}(n)=\lambda_{\phi}(n)\rho_{\phi}(1),
\end{equation}
where $\lambda_{\phi}(1)=1$ and $\lambda_{\phi}(n)$ are the Hecke eigenvalues. We know from \cite{KS} that
\begin{equation}\label{e:ks}
\lambda_{\phi}(n)\ll_{\epsilon}n^{\theta+\epsilon}
\end{equation}
with $\theta=\frac{7}{64}$.
Define
\[
\Phi(z)=\frac{\phi(z)}{\rho_{\phi}(1)e^{-\frac{\pi}{2}t_{\phi}}\sqrt{y}}=\sum_{n\not=0}\lambda_{\phi}(n)e^{\frac{\pi}{2}t_{\phi}}\mathrm{K}_{it_{\phi}}(2\pi|n|y)e(nx).
\]

\begin{proposition}\label{p:approx} Let $\Delta=t_{\phi}^{\frac{1}{3}}\log t_{\phi}$. Then, for any $c<y\leq\frac{t_{\phi}}{2\pi}$ with any fixed $c>0$., we have
$(1)$ For $X\ge 1$ and $k=1,2$,
\[
\sum_{|n|\le X}\left|\lambda_{\phi}(n)\right|^{2k}\ll_{\epsilon}Xt_{\phi}^{\epsilon}.
\]
$(2)$
\begin{equation*}\begin{split}
\int_{-\frac{1}{2}}^{\frac{1}{2}}\Phi(x+iy)^2dx=&2\pi\sum_{|n|\le\frac{t_{\phi}-\Delta}{2\pi y}}\frac{\left|\lambda_{\phi}(n)\right|^2}
{\left(t_{\phi}^2-(2\pi ny)^2\right)^{\frac{1}{2}}}
\sin^2\left(\frac{\pi}{4}+t_{\phi}\mathrm{H}\left(\frac{2\pi ny}{t_{\phi}}\right)\right)\\
&+O\left(t_{\phi}^{2\theta-\frac{2}{3}+\epsilon}\left(\frac{\Delta}{y}+1\right)\right).
\end{split}
\end{equation*}

\end{proposition}

Set
\[
\widetilde{\rho}_{\phi}(n)=\rho_{\phi}(n)e^{-\frac{\pi}{2}t_{\phi}}.
\]

\begin{proposition}\label{p:rhobounds} For any fixed $\omega>0$,
\[
\sum_{10^{-5}\omega t_{\phi}\le|n|\le\omega t_{\phi}}\left|\widetilde{\rho}_{\phi}(n)\right|^2\ge\frac{1}{1000}\omega t_{\phi}
\]
for sufficiently large $t_{\phi}$.
\end{proposition}

We note that this proposition crucially relies on the QUE result of Lindenstrauss \cite{Lin} and Soundararajan \cite{Sound} .

\begin{thms}\label{t:GRS} If $\beta$ is a long enough but fixed compact subsegment in $\delta_1$, then for $\epsilon>0$
\[
1\ll_{\beta}\int_{\beta}\phi^2(z)\frac{dy}{y}\le\int_{\delta}\phi^2(z)\frac{dy}{y}\ll_{\epsilon}t_{\phi}^{\epsilon}.
\]
\end{thms}

For Propositions \ref{p:approx}, \ref{p:rhobounds} and Theorem \ref{t:GRS}, see \cite[p. 1529, p. 1541,  p. 1531 and p. 1520]{GRS}.

We introduce a modified theorem of Littlewood \cite[p. 336, Theorem 1 (i)]{Lit} below.

Let $f$ be a real valued continuous function on the interval $[a,b+\eta]$ where $0<a<b$ and $0<\eta<(b-a)/10^7$. Define $M_{\lambda}(f)$ by
\[
M_{\lambda}(f)=\left(\frac{1}{b-a}\int_{a}^{b}|f(y)|^{\lambda}dy\right)^{1/\lambda}\qquad(\lambda>0).
\]
Let $N(f)$ be the number of sign changes of $f$
on the interval $[a,b]$. Define $J(f,\eta)$ by
\[
J(f,\eta)=\frac{1}{b-a}\int_{a}^{b}\left|\int_0^{\eta}f(y+v)dv\right|dy.
\]

\begin{thms}[Littlewood]\label{t:lit} Let $f$ be a real valued continuous function on $[a,b+\eta]$ that satisfies
\[
M_1(f)\ge cM_2(f)\qquad(0<c\le1),
\]
where $\eta=\omega(b-a)/N$ with $\omega>0$ and $N>10^7(\omega+7)$. Assume
\[
J(f,\eta)<\frac{1}{16}c^3\eta M_2(f).
\]
Then, we have
\[
N(f)\ge\frac{c^2}{10(\omega+2)}N.
\]

\end{thms}

\begin{proof} We follow the proof in \cite[pp. 339-340]{Lit}. To do so, we need to have appropriate settings as follows. Define
\begin{equation*}\begin{split}
J^*(f,\eta)&=\frac{1}{b-a}\int_{a}^{b}\int_0^{\eta}\left|f(y+v)\right|dvdy;\\
J_m(f,\eta)&=\frac{1}{b-a}\int_{a+\frac{(b-a)m}{N}}^{a+\frac{(b-a)(m+1)}{N}}\left|\int_0^{\eta}f(y+v)dv\right|dy;\\
J_m^*(f,\eta)&=\frac{1}{b-a}\int_{a+\frac{(b-a)m}{N}}^{a+\frac{(b-a)(m+1)}{N}}\int_0^{\eta}\left|f(y+v)\right|dvdy.
\end{split}
\end{equation*}
Set
\begin{equation*}\begin{split}
I_m&=\left\{y:a+\frac{(b-a)m}{N}\le y\le a+\frac{(b-a)m}{N}+\frac{\omega(b-a)+(b-a)}{N}\right\};\\
\mathcal{M}_1&=\left\{1\le m\le N^*:f\text{ has a sign change in }I_m\right\};\\
\mathcal{M}_2&=\{m: m=1,2,\ldots,N^*\}\setminus\mathcal{M}_1,
\end{split}
\end{equation*}
where
\[
N^*=[N-\omega-1].
\]
Then, we have
\begin{equation*}\begin{split}
N(f)&\ge\frac{\left|\mathcal{M}_1\right|}{\omega+2};\\
J_{m_1}(f,\eta)&\le J_{m_1}^*(f,\eta)\qquad(m_1\in\mathcal{M}_1);\\
J_{m_2}(f,\eta)&=J_{m_2}^*(f,\eta)\qquad(m_2\in\mathcal{M}_2);\\
\sum_{m\in\mathcal{M}_2}J_m^*(f,\eta)&=\sum_{m\in\mathcal{M}_2}J_m(f,\eta)\le J(f,\eta)<\frac{c^3}{16}\eta M_2(f).
\end{split}
\end{equation*}

We state the following crucial lemma as in \cite[p. 340]{Lit}.

\begin{lemma} Let $H$ be the set in $(a,b)$ in which $|f|\ge\frac{1}{2}cM_2(f)$. Then, $|H|\ge\frac{b-a}{2}c^2$.
\end{lemma}

The proof of Theorem \ref{t:lit} is identical as in \cite[pp. 339-340]{Lit}. We omit the proof.
\end{proof}

\section{Proof of Theorem \ref{t:l4}}

By \eqref{e:rho}, \eqref{e:lambda} and \eqref{e:ks}, we have
\[
\rho_{\phi}(n)\ll_{\epsilon}e^{\frac{\pi}{2}t_{\phi}}n^{\epsilon+\frac{7}{64}}.
\]
Thus, we get
\begin{equation}\label{e:trun}
\sum_{|n|>t_{\phi}}\rho_{\phi}(n)e(nx)\mathrm{K}_{it_{\phi}}(2\pi |n|y)\ll e^{-t_{\phi}}\qquad\left(y\ge\frac{1}{2}\right).
\end{equation}
Thus, we have
\begin{equation}\label{e:phiu}
\frac{1}{\left(\rho_{\phi}(1)e^{-\frac{\pi}{2}t_{\phi}}\right)^4}\int_{\mathbb{X}}\phi^4(z)\frac{dxdy}{y^2}\ll1+\int_{-1/2}^{1/2}\int_{1/2}^{\infty}\left|\psi(z)\right|^4dydx,
\end{equation}
where
\[
\psi(z)=\sum_{1\le n<t_{\phi}}\lambda_{\phi}(n)e(nx)e^{\frac{\pi}{2}t_{\phi}}\mathrm{K}_{it_{\phi}}(2\pi ny).
\]
Let $l_{\epsilon}$ be a positive integer such that
\[
1-\epsilon(l_{\epsilon}+1)\ge\epsilon+\frac{1}{3}>1-\epsilon(l_{\epsilon}+2).
\]
Define
\begin{equation*}\begin{split}
\psi_0(z)=&\sum_{n<\epy}\lambda_{\phi}(n)e(nx)e^{\frac{\pi}{2}t_{\phi}}\mathrm{K}_{it_{\phi}}(2\pi ny),\\
\psi_{1l}(z)=&\sum_{\frac{t_{\phi}-t_{\phi}^{1-l\epsilon}}{2\pi y}\le n<\frac{t_{\phi}-t_{\phi}^{1-(l+1)\epsilon}}{2\pi y}}\lambda_{\phi}(n)e(nx)e^{\frac{\pi}{2}t_{\phi}}\mathrm{K}_{it_{\phi}}(2\pi ny)\qquad(1\le l\le l_{\epsilon}),\\
\psi_{1\epsilon}(z)=&\sum_{\frac{t_{\phi}-t_{\phi}^{1-(l_{\epsilon}+1)\epsilon}}{2\pi y}\le n<\frac{t_{\phi}-t_{\phi}^{\epsilon+\frac{1}{3}}}{2\pi y}}\lambda_{\phi}(n)e(nx)e^{\frac{\pi}{2}t_{\phi}}\mathrm{K}_{it_{\phi}}(2\pi ny),\\
\psi_2(z)=&\sum_{\frac{t_{\phi}-t_{\phi}^{\epsilon+\frac{1}{3}}}{2\pi y}\le n<\frac{t_{\phi}+t_{\phi}^{\epsilon+\frac{1}{3}}}{2\pi y}}\lambda_{\phi}(n)e(nx)e^{\frac{\pi}{2}t_{\phi}}\mathrm{K}_{it_{\phi}}(2\pi ny),\\
\psi_3(z)=&\sum_{\frac{t_{\phi}+t_{\phi}^{\epsilon+\frac{1}{3}}}{2\pi y}\le n<t_{\phi}}\lambda_{\phi}(n)e(nx)e^{\frac{\pi}{2}t_{\phi}}\mathrm{K}_{it_{\phi}}(2\pi ny).
\end{split}
\end{equation*}
Thus,
\[
\psi(z)=\psi_0(z)+\sum_{1\le l\le l_{\epsilon}}\psi_{1l}(z)+\psi_{1\epsilon}(z)+\psi_2(z)+\psi_3(z).
\]
Hence, by \eqref{e:rho} and \eqref{e:phiu}, it suffices to show that
\begin{equation}\label{e:ub}
\int_{-1/2}^{1/2}\int_{1/2}^{\infty}\left|f(z)\right|^4dydx\ll_{\epsilon} t_{\phi}^{10\epsilon}
\end{equation}
for $f(z)=\psi_0(z)$, $\psi_{1l}(z)$ $(1\le l\le l_{\epsilon})$, $\psi_{1\epsilon}(z)$, $\psi_2(z)$, $\psi_3(z)$.
Set
\begin{equation*}\begin{split}
\psi_{21}(z)=&\sum_{\frac{t_{\phi}-t_{\phi}^{\frac{1}{3}}}{2\pi y}\le n<\frac{t_{\phi}+t_{\phi}^{\frac{1}{3}}}{2\pi y}}\lambda_{\phi}(n)e(nx)e^{\frac{\pi}{2}t_{\phi}}\mathrm{K}_{it_{\phi}}(2\pi ny),\\
\psi_{22}(z)=&\psi_2(z)-\psi_{21}(z).
\end{split}
\end{equation*}
By Lemma \ref{l:kbessel} (3),
\[
\psi_{21}(z)\ll\sum_{\frac{t_{\phi}-t_{\phi}^{\frac{1}{3}}}{2\pi y}\le n<\frac{t_{\phi}+t_{\phi}^{\frac{1}{3}}}{2\pi y}}\left|\lambda_{\phi}(n)\right|\cdot\frac{1}{t_{\phi}^{\frac{1}{3}}}.
\]
By the H\"older inequality and Lemma \ref{l:kbessel} (1),
\begin{equation*}\begin{split}
\int_{1/2}^{\infty}\psi_{21}^4(z)\ll&\int_{1/2}^{\infty}\frac{1}{t_{\phi}^{\frac{4}{3}}}\Biggl(\sum_{\frac{t_{\phi}-t_{\phi}^{\frac{1}{3}}}{2\pi y}\le n<\frac{t_{\phi}+t_{\phi}^{\frac{1}{3}}}{2\pi y}}1^{\frac{4}{3}}\Biggr)^3\sum_{\frac{t_{\phi}-t_{\phi}^{\frac{1}{3}}}{2\pi y}\le n<\frac{t_{\phi}+t_{\phi}^{\frac{1}{3}}}{2\pi y}}\left|\lambda_{\phi}(n)\right|^4dy\\
\ll&\frac{1}{t_{\phi}^{\frac{1}{3}}}\sum_{n\le t_{\phi}}\left|\lambda_{\phi}(n)\right|^4\int_{\max\bigl(\frac{1}{2},\frac{t_{\phi}-t_{\phi}^{\frac{1}{3}}}{2\pi n}\bigr)}^{\frac{t_{\phi}+t_{\phi}^{\frac{1}{3}}}{2\pi n}}\frac{1}{y^3}dy\\
\ll&\frac{1}{t_{\phi}}\sum_{n\le t_{\phi}}\left|\lambda_{\phi}(n)\right|^4\\
\ll& t_{\phi}^{\epsilon};
\end{split}
\end{equation*}
Define
\[
E=\biggl\{n\in\mathbb{Z}:\frac{t_{\phi}-t_{\phi}^{\epsilon+\frac{1}{3}}}{2\pi y}\le n\le\frac{t_{\phi}-t_{\phi}^{\frac{1}{3}}}{2\pi y}\text{ or }
\frac{t_{\phi}+t_{\phi}^{\frac{1}{3}}}{2\pi y}\le n\le\frac{t_{\phi}+t_{\phi}^{\epsilon+\frac{1}{3}}}{2\pi y}\biggr\}.
\]
By Lemma \ref{l:kbessel} and the H\"older inequality,
\begin{equation*}\begin{split}
\int_{1/2}^{\infty}\psi_{22}(z)^4\ll&\int_{1/2}^{\infty}\Biggl(\sum_{n\in E}\frac{\left|\lambda_{\phi}(n)\right|}{\left|t_{\phi}^2-(2\pi ny)^2\right|^{\frac{1}{4}}}\Biggr)^4dy\\
\ll&\frac{1}{t_{\phi}t_{\phi}^{\frac{1}{3}}}\int_{1/2}^{\infty}\left(\sum_{n\in E}\left|\lambda_{\phi}(n)\right|\right)^4dy\\
\ll&\frac{1}{t_{\phi}t_{\phi}^{\frac{1}{3}}}\cdot t_{\phi}^{3\epsilon+1}\sum_{n\le t_{\phi}}\left|\lambda_{\phi}(n)\right|^4\frac{t_{\phi}^{\frac{1}{3}+\epsilon}}{n}\\
\ll& t_{\phi}^{5\epsilon}.
\end{split}
\end{equation*}
Thus, we prove \eqref{e:ub} for $f(z)=\psi_2(z)$. Also, similarly, \eqref{e:ub} can be justified for $f(z)=\psi_{1\epsilon}(z)$.

For convenience, define
\[
H(n,y)=t_{\phi}\textrm{H}\left(\frac{2\pi ny}{t_{\phi}}\right).
\]
By the facts that $\textrm{H}$ is decreasing in $(0,1)$ and increasing in $[1,\infty)$ and $\textrm{H}$ satisfies
\[
\textrm{H}(\xi)\asymp\left|\xi^2-1\right|^{\frac{3}{2}}\qquad\text{near }1,
\]
we have
\[
H(n,y)\gg t_{\phi}^{1-\frac{3}{2}\epsilon}.
\]
for any $2\pi ny<t_{\phi}-t_{\phi}^{1-\epsilon}$. From this and Lemma \ref{l:kbessel} (1), we have
\[
\psi_0(z)\ll\left|\psi_{01}(z)\right|+\left|\psi_{02}(z)\right|,
\]
where
\begin{equation*}\begin{split}
\psi_{01}(z)=&\sum_{n<\epy}\frac{\lambda_{\phi}(n)e(nx)}{\left(t_{\phi}^2-(2\pi ny)^2\right)^{\frac{1}{4}}}\sin\left[\frac{\pi}{4}+H(n,y)\right],\\
\psi_{02}(z)=&\sum_{n<\epy}\frac{\left|\lambda_{\phi}(n)\right|}{t_{\phi}^{\frac{1}{2}}t_{\phi}^{1-2\epsilon}}.
\end{split}
\end{equation*}
By Proposition \ref{p:approx} (1) together with the Cauchy–Schwarz inequality,
\begin{equation}\label{e:psi02}
\int_{1/2}^{\infty}\left|\psi_{02}(z)\right|^4dy\ll\int_{1/2}^{\infty}\left(t_{\phi}^{-1/2+3\epsilon} y^{-1}\right)^4dy\ll1.
\end{equation}
Set
\begin{equation*}\begin{split}
\psi_{01}^+(z)=&\sum_{n<\epy}\frac{\lambda_{\phi}(n)e(nx)}{\left(t_{\phi}^2-(2\pi ny)^2\right)^{\frac{1}{4}}}e^{iH(n,y)},\\
\psi_{01}^-(z)=&\sum_{n<\epy}\frac{\lambda_{\phi}(n)e(nx)}{\left(t_{\phi}^2-(2\pi ny)^2\right)^{\frac{1}{4}}}e^{-iH(n,y)}.
\end{split}
\end{equation*}
For $\mathbf{n}=(n_1,\ldots,n_4)$ with $2\pi n_jy<t_{\phi}$ and $y\ge1/2$ ($j=1,\ldots,4$), $\mathrm{D}(\mathbf{n},y)$ and $\mathrm{d}(\mathbf{n},y)$ denote
\[
\mathrm{D}(\mathbf{n},y)=H(n_1,y)+H(n_2,y)
-H(n_3,y)-H(n_4,y),
\]
\[
\mathrm{d}(\mathbf{n},y)=\frac{\partial}{\partial y}\left(\mathrm{D}(\mathbf{n},y)\right).
\]
Put
\[
\mathcal{N}=\biggl\{\mathbf{n}=(n_1,\ldots,n_4):n_1+n_2=n_3+n_4,n_1\not=n_3,n_4,1\le n_1,\ldots,n_4\le\frac{t_{\phi}-t_{\phi}^{1-\epsilon}}{\pi}\biggr\}.
\]
For $\mathbf{n}=(n_1,n_2,n_3,n_4)\in\mathcal{N}$, set
\begin{equation*}\begin{split}
\beta(\mathbf{n})=&\max\left\{\frac{t_{\phi}-t_{\phi}^{1-\epsilon}}{2\pi n_1},\cdots,\frac{t_{\phi}-t_{\phi}^{1-\epsilon}}{2\pi n_4}\right\},\\
\beta_0(m,n)=&\max\left\{\frac{t_{\phi}-t_{\phi}^{1-\epsilon}}{2\pi m},\frac{t_{\phi}-t_{\phi}^{1-\epsilon}}{2\pi n}\right\},\\
Q(\mathbf{n},y)=&\frac{1}{\left(t_{\phi}^2-(2\pi n_1y)^2\right)^{\frac{1}{4}}}\cdots\frac{1}{\left(t_{\phi}^2-(2\pi n_4y)^2\right)^{\frac{1}{4}}}\\
\mathrm{I}(m,n)=&\int_{1/2}^{\beta_0(m,n)}\frac{1}{\left(t_{\phi}^2-(2\pi my)^2\right)^{\frac{1}{2}}\left(t_{\phi}^2-(2\pi ny)^2\right)^{\frac{1}{2}}}dy\\
\mathrm{I}(\mathbf{n})=&\int_{1/2}^{\beta(\mathbf{n})}Q(\mathbf{n},y)e^{i\mathrm{D}(\mathbf{n},y)}dy.
\end{split}
\end{equation*}
We have
\begin{equation}\label{e:psi+01}
\int_{-1/2}^{1/2}\int_{1/2}^{\infty}\left|\psi_{01}(z)\right|^4dydx\ll\int_{-1/2}^{1/2}\int_{1/2}^{\infty}\left|\psi_{01}^+(z)\right|^4+\left|\psi_{01}^-(z)\right|^4dydx
\ll\mathrm{I}_0+\mathrm{I}_1,
\end{equation}
\begin{equation*}\begin{split}
\mathrm{I}_0=&\sum_{k<\frac{2\left(t_{\phi}-t_{\phi}^{1-\epsilon}\right)}{\pi}}\sum_{\substack{m+n=k\\m,n\ge1}}\lambda_{\phi}(m)^2\lambda_{\phi}(n)^2\mathrm{I}(m,n),\\
\mathrm{I}_1=&\sum_{\mathbf{n}=(n_1,\ldots,n_4)\in\mathcal{N}}\lambda_{\phi}(n_1)\cdots\lambda_{\phi}(n_4)\mathrm{I}(\mathbf{n}).
\end{split}
\end{equation*}
For $m<n$,
\[
\mathrm{I}(m,n)\ll\frac{1}{t_{\phi}}\int_{1/2}^{\frac{t_{\phi}-t_{\phi}^{1-\epsilon}}{2\pi n}}\frac{dy}{\left(t_{\phi}-2\pi my\right)^{\frac{1}{2}}\left(t_{\phi}-2\pi ny\right)^{\frac{1}{2}}}\ll\frac{\log t_{\phi}}{nt_{\phi}}.
\]
By this and Proposition \ref{p:approx} (1),
\begin{equation}\label{e:I0}\begin{split}
\mathrm{I}_0\ll&\sum_{k<\frac{2\left(t_{\phi}-t_{\phi}^{1-\epsilon}\right)}{\pi}}\,\,\,\sum_{\substack{m\le n\\ m\le\frac{k}{2},\, m+n=k}}\lambda_{\phi}(m)^2\lambda_{\phi}(n)^2
\cdot\frac{\log t_{\phi}}{nt_{\phi}}\\
\ll&\frac{\log t_{\phi}}{t_{\phi}}\sum_{k<t_{\phi}}\sum_{\substack{m\le n\\ m\le\frac{k}{2},\, m+n=k}}\frac{\lambda_{\phi}(m)^4+\lambda_{\phi}(n)^4}{n}\\
\ll_{\epsilon}&\,t_{\phi}^{\epsilon}.
\end{split}
\end{equation}

\begin{lemma}\label{l:h} For $\mathbf{n}=(n_1,\ldots,n_4)\in\mathcal{N}$,

$(1)$
  \[
  \frac{d(\mathbf{n},y)}{n_3n_4-n_1n_2}>0;
  \]

$(2)$ For $2\pi n_1y,\ldots,2\pi n_4y<t_{\phi}-t_{\phi}^{1-\epsilon}$,
  \[
  d(\mathbf{n},y)\gg\frac{y\left|n_1n_2-n_3n_4\right|}{t_{\phi}};
  \]

$(3)$ For $2\pi n_1y,\ldots,2\pi n_4y<t_{\phi}$,
  \[
  \frac{\partial d(\mathbf{n},y)/\partial y}{d(\mathbf{n},y)}>0.
  \]
\end{lemma}

\begin{proof}[Proof of Lemma \ref{l:h}] (1) and (2). Observe
\[
-d(\mathbf{n},y)=\frac{h_1+h_2-h_3-h_4}{y},
\]
where
\[
h_j=\sqrt{t_{\phi}^2-\left(2\pi n_jy\right)^2}\qquad(j=1,\ldots,4).
\]
Using $n_1+n_2=n_3+n_4$ and so $n_3^2+n_4^2=n_1^2+n_2^2+2n_1n_2-2n_3n_4$, we get
\begin{equation*}\begin{split}
h_1+h_2-h_3-h_4=&\frac{h_1^2+h_2^2-h_3^2-h_4^2+2h_1h_2-2h_3h_4}{h_1+h_2+h_3+h_4},\\
\frac{h_1^2+h_2^2-h_3^2-h_4^2}{(2\pi y)^2(n_1n_2-n_3n_4)}=&\frac{n_3^2+n_4^2-n_1^2-n_2^2}{n_1n_2-n_3n_4}=2,\\
\frac{h_1h_2-h_3h_4}{(2\pi y)^2(n_1n_2-n_3n_4)}=&\frac{t_{\phi}^2\left(n_3^2+n_4^2-n_1^2-n_2^2\right)+(2\pi y)^2(n_1^2n_2^2-n_3^2n_4^2)}{(n_1n_2-n_3n_4)(h_1h_2+h_3h_4)}>0,
\end{split}
\end{equation*}
especially
\begin{equation}\label{e:ch}
\left|h_1+h_2-h_3-h_4\right|\ge\frac{4(2\pi y)^2\left|n_1n_2-n_3n_4\right|t_{\phi}^2}{\left(h_1h_2+h_3h_4\right)\left(h_1+h_2+h_3+h_4\right)}.
\end{equation}

From this, (1) and (2) follow.

(3) We have
\[
\frac{\partial}{\partial y}\left(d(\mathbf{n},y)\right)=\frac{t_{\phi}^2}{y^2}\left(\frac{1}{h_1}+\frac{1}{h_2}-\frac{1}{h_3}-\frac{1}{h_4}\right),
\]
\[
\left(h_2h_3h_4+h_1h_3h_4\right)^2-\left(h_1h_2h_4+h_1h_2h_3\right)^2=H_1+\frac{2h_1h_2h_3h_4}{h_3h_4+h_1h_2}H_2,
\]
where
\[
H_1=h_2^2h_3^2h_4^2+h_1^2h_3^2h_4^2-h_1^2h_2^2h_4^2-h_1^2h_2^2h_3^2,
\qquad
H_2=h_3^2h_4^2-h_1^2h_2^2.
\]
Using $n_1+n_2=n_3+n_4$ and so $n_3^2+n_4^2=n_1^2+n_2^2+2n_1n_2-2n_3n_4$, we obtain
\[
H_2=\left(n_3n_4-n_1n_2\right)(2\pi y)^2\left(2t_{\phi}^2+(2\pi y)^2(n_3n_4+n_1n_2)\right),
\]
\begin{multline*}
H_1=(2\pi y)^2\left(n_3n_4-n_1n_2\right) \biggl[2\left(t_{\phi}^4-(2\pi n_1y)^2(2\pi n_2y)^2\right)
\\+\left(n_1n_2+n_3n_4\right)(2\pi y)^2\left(2t_{\phi}^2-(2\pi y)^2\left(n_1^2+n_2^2\right)\right)\biggr].
\end{multline*}
Thus,
\[
\frac{\partial d(\mathbf{n},y)/\partial y}{n_3n_4-n_1n_2}>0.
\]
Since $d(\mathbf{n},y)/(n_3n_4-n_1n_2)>0$, (3) follows.
\end{proof}

We have
\[
\mathrm{I}(\mathbf{n})=\mathrm{I}_1(\mathbf{n},y)\big|_{1/2}^{\beta(\mathbf{n})}-\int_{1/2}^{\beta(\mathbf{n})}\mathrm{I}_2(\mathbf{n},y)dy,
\]
where
\begin{equation*}\begin{split}
\mathrm{I}_1(\mathbf{n},y)=\frac{Q(\mathbf{n},y)}{id(\mathbf{n},y)}e^{i\mathrm{D}(\mathbf{n},y)},\qquad
\mathrm{I}_2(\mathbf{n},y)=\frac{\partial Q(\mathbf{n},y)/\partial y \cdot d(\mathbf{n},y)-Q(\mathbf{n},y)\cdot \partial d(\mathbf{n},y)/\partial y}{id(\mathbf{n},y)^2}e^{i\mathrm{D}(\mathbf{n},y)}.
\end{split}
\end{equation*}
By Lemma \ref{l:h} and the fact that
\[
Q(\mathbf{n},y)\ll\frac{1}{t_{\phi}^{2-\epsilon}},
\]
\[
\mathrm{I}_1(\mathbf{n},y)\ll\frac{1}{t_{\phi}^{2-\epsilon}}\cdot\frac{t_{\phi}}{y\left|n_1n_2-n_3n_4\right|}
\ll\frac{t_{\phi}^{\epsilon}}{y\left|n_1n_2-n_3n_4\right|t_{\phi}}.
\]
By Lemma \ref{l:h} and the facts that $Q(\mathbf{n},y),\,\partial Q(\mathbf{n},y)/\partial y>0$, $Q(\mathbf{n},y)\ll1/t_{\phi}^{2-\epsilon}$,
\begin{equation*}\begin{split}
\int_{1/2}^{\beta(\mathbf{n})}\mathrm{I}_2(\mathbf{n},y)dy\ll&\frac{n_3n_4-n_1n_2}{\left|n_1n_2-n_3n_4\right|}\int_{1/2}^{\beta(\mathbf{n})}2\frac{\partial Q(\mathbf{n},y)/\partial y}{d(\mathbf{n},y)}
-\frac{\partial}{\partial y}\left(\frac{Q(\mathbf{n},y)}{d(\mathbf{n},y)}\right)dy\\
\ll&\frac{t_{\phi}^{1+\epsilon}}{\left|n_1n_2-n_3n_4\right|}\int_{1/2}^{\beta(\mathbf{n})}\frac{\partial Q(\mathbf{n},y)/\partial y}{y}dy+\left|\frac{Q(\mathbf{n},1/2)}{d(\mathbf{n},1/2)}\right|
+\left|\frac{Q(\mathbf{n},\beta)}{d(\mathbf{n},\beta)}\right|\\
\ll&\frac{1}{\left|n_1n_2-n_3n_4\right|}\cdot\frac{1}{t_{\phi}^{1-2\epsilon}}+\frac{t_{\phi}^{1+\epsilon}}{\left|n_1n_2-n_3n_4\right|}
\int_{1/2}^{\beta(\mathbf{n})}\frac{Q(\mathbf{n},y)}{y^2}dy\\
\ll&\frac{1}{\left|n_1n_2-n_3n_4\right|}\cdot\frac{1}{t_{\phi}^{1-2\epsilon}}.
\end{split}
\end{equation*}
By these, we get
\[
\mathrm{I}(\mathbf{n})\ll\frac{1}{\left|n_1n_2-n_3n_4\right|}\cdot\frac{1}{t_{\phi}^{1-2\epsilon}},
\]
By this and the Cauchy-Schwartz inequality,
\[
\mathrm{I}_1\ll\frac{t_{\phi}^{2\epsilon}}{t_{\phi}}\sum_{\mathbf{n}=(n_1,\ldots,n_4)\in\mathcal{N}}
\frac{\left|\lambda_{\phi}(n_1)\cdots\lambda_{\phi}(n_4)\right|}{\left|n_1n_2-n_3n_4\right|}
\ll\mathrm{I}_{11}+\ldots+\mathrm{I}_{14},
\]
where
\[
\mathrm{I}_{1j}=\frac{t_{\phi}^{2\epsilon}}{t_{\phi}}\sum_{\mathbf{n}=(n_1,\ldots,n_4)\in\mathcal{N}}
\frac{\left|\lambda_{\phi}(n_j)\right|^4}{\left|n_1n_2-n_3n_4\right|}\qquad(j=1,\ldots,4).
\]
By Lemma \ref{l:kbessel} (1) and the fact that
\begin{equation*}\begin{split}
\left|n_1n_2-n_3n_4\right|=&\left|(n_a-n_3)(n_a-n_4)\right|=\left|(n_b-n_1)(n_b-n_2)\right|\qquad(a=1,2,\,b=3,4),
\end{split}
\end{equation*}
we can demonstrate that for $j=1,2$,
\begin{equation*}\begin{split}
\mathrm{I}_{1j}=&\frac{t_{\phi}^{2\epsilon}}{t_{\phi}}\sum_{\mathbf{n}=(n_1,\ldots,n_4)\in\mathcal{N}}
\frac{\left|\lambda_{\phi}(n_j)\right|^4}{\left|(n_j-n_3)(n_j-n_4)\right|}\\
\le&\frac{t_{\phi}^{2\epsilon}}{t_{\phi}}\sum_{n_j<t_{\phi}}\left|\lambda_{\phi}(n_j)\right|^4\sum_{\substack{n_3,n_4<t_{\phi}\\ n_3,n_4\not=n_j}}\frac{1}{\left|(n_j-n_3)(n_j-n_4)\right|}\\
\ll&\frac{t_{\phi}^{2\epsilon}}{t_{\phi}}\sum_{n_j<t_{\phi}}\left|\lambda_{\phi}(n_j)\right|^4(\log t_{\phi})^2\\
\ll&t_{\phi}^{3\epsilon}
\end{split}
\end{equation*}
and similarly,
\[
\mathrm{I}_{1j}\ll t_{\phi}^{3\epsilon}
\]
for $j=3,4$.
By this, \eqref{e:I0} and \eqref{e:psi+01},
\[
\int_{-1/2}^{1/2}\int_{1/2}^{\infty}\left|\psi_{01}(z)\right|^4dydx\ll_{\epsilon}t_{\phi}^{3\epsilon}.
\]
By this and \eqref{e:psi02}, we have
\[
\int_{-1/2}^{1/2}\int_{1/2}^{\infty}\left|\psi_{0}(z)\right|^4dydx\ll_{\epsilon}t_{\phi}^{3\epsilon}.
\]
Thus, we prove \eqref{e:ub} for $f(z)=\psi_0(z)$.

We now treat the remaining cases separately, due to the transitional behavior of $\mathrm{K}_{ir}(u)$ (see, for instance, Lemma \ref{l:kbessel}(3)).

Find a positive integer $K$ such that $\epsilon K>7$. Recall
\[
H(n,y)=t_{\phi}\mathrm{H}\left(\frac{2\pi ny}{t_{\phi}}\right).
\]
By Lemma \ref{l:km}, \eqref{e:ks} and the facts that
\[
H(n,y)^K\ge H\left(n,t_{\phi}-t_{\phi}^{\epsilon+\frac{1}{3}}\right)^K\gg\left(t_{\phi}\left(\frac{t_{\phi}^{\epsilon+\frac{1}{3}}}{t_{\phi}}\right)^{\frac{3}{2}}\right)^K
=t_{\phi}^{\frac{3}{2}\epsilon K}>t_{\phi}^7
\]
and
\[
\int_{1/2}^{\infty}\sum_{\substack{n_1+n_2=n_3+n_4\\ \frac{t_{\phi}-t_{\phi}^{1-l\epsilon}}{2\pi y}\le n_1,n_2,n_3,n_4<\frac{t_{\phi}-t_{\phi}^{1-(l+1)\epsilon}}{2\pi y}}}
\left(\frac{t_{\phi}^{\frac{8}{64}}}{t_{\phi}t_{\phi}^{\frac{1}{3}}}\right)^4dy\ll1,
\]
we have
\[
\int_{-1/2}^{1/2}\int_{1/2}^{\infty}\left|\psi_{1l}(z)\right|^4dydx\ll\int_{-1/2}^{1/2}\int_{1/2}^{\infty}\left|\widetilde{\psi}_{1l}(z)\right|^4dydx+1,
\]
where
\[
\widetilde{\psi}_{1l}(z)=\sum_{\frac{t_{\phi}-t_{\phi}^{1-l\epsilon}}{2\pi y}\le n<\frac{t_{\phi}-t_{\phi}^{1-(l+1)\epsilon}}{2\pi y}}
\frac{\lambda_{\phi}(n)e(nx)}{\left(t_{\phi}^2-(2\pi ny)^2\right)^{\frac{1}{4}}}\\
e^{iH(n,y)}
\sum_{k<K}\frac{a_k}{H(n,y)^k}.
\]
Put
\begin{equation*}\begin{split}
\mathcal{N}_{1l}=&\biggl\{(n_1,\ldots,n_4):n_1+n_2=n_3+n_4,n_1\not=n_3,n_4,1\le n_1,\ldots,n_4\le\frac{t_{\phi}-t_{\phi}^{1-(l+1)\epsilon}}{\pi}\biggr\}\\
\mathcal{N}_{2l}=&\bigg\{(m,n):1\le m,n\le\frac{t_{\phi}-t_{\phi}^{1-(l+1)\epsilon}}{\pi}\biggr\}.
\end{split}
\end{equation*}
For $\mathbf{n}=(n_1,\ldots,n_4)\in\mathcal{N}_{1l}$, $(m,n)\in\mathcal{N}_{2l}$ and $\mathbf{k}=(k_1,\ldots,k_4)$ with $k_1,\ldots,k_4<K$, set
\begin{equation*}\begin{split}
\alpha=&\max\biggl\{\frac{1}{2},\frac{t_{\phi}-t_{\phi}^{1-l\epsilon}}{2\pi n_1},\ldots,\frac{t_{\phi}-t_{\phi}^{1-l\epsilon}}{2\pi n_4}\biggr\},\,\,\,
\beta=\min\biggl\{\frac{t_{\phi}-t_{\phi}^{1-(l+1)\epsilon}}{2\pi n_1},\ldots,\frac{t_{\phi}-t_{\phi}^{1-(l+1)\epsilon}}{2\pi n_4}\biggr\},\\
\alpha_0=&\max\biggl\{\frac{1}{2},\frac{t_{\phi}-t_{\phi}^{1-l\epsilon}}{2\pi m},\frac{t_{\phi}-t_{\phi}^{1-l\epsilon}}{2\pi n}\biggr\},\,\,\,
\beta_0=\min\biggl\{\frac{t_{\phi}-t_{\phi}^{1-(l+1)\epsilon}}{2\pi m},\frac{t_{\phi}-t_{\phi}^{1-(l+1)\epsilon}}{2\pi n}\biggr\},
\end{split}
\end{equation*}
\[
R(\mathbf{n},\mathbf{k},y)=\frac{Q(\mathbf{n},y)}{H(n_1,y)^{k_1}\cdots H(n_4,y)^{k_4}},
\]
\[
\mathrm{I}(\mathbf{n},\mathbf{k})=\int_{\alpha}^{\beta}R(\mathbf{n},\mathbf{k},y)e^{iH(\mathbf{n},y)}dy,
\]
\[
\mathrm{I}(m,n,\mathbf{k})=\int_{\beta_0}^{\alpha_0}\frac{dy}{\left(t_{\phi}^2-(2\pi my)^2\right)^{\frac{1}{2}}\left(t_{\phi}^2-(2\pi ny)^2\right)^{\frac{1}{2}}
\cdot H(m,y)^{k_1+k_2}H(n,y)^{k_3+k_4}}.
\]
We have
\[
\int_{-1/2}^{1/2}\int_{1/2}^{\infty}\left|\widetilde{\psi}_{1l}(z)\right|^4dydx=\sum_{\substack{\mathbf{k}=(k_1,k_2,k_3,k_4)\\ k_1,\ldots,k_4<K}}\left(\mathrm{I}_{1l}(\mathbf{k})+\mathrm{I}_{2l}(\mathbf{k})\right),
\]
where
\begin{equation*}\begin{split}
\mathrm{I}_{1l}(\mathbf{k})=&\sum_{\mathbf{n}=(n_1,\ldots,n_4)\in\mathcal{N}_{1l}}\lambda_{\phi}(n_1)\cdots\lambda_{\phi}(n_4)\mathrm{I}(\mathbf{n},\mathbf{k})\\
\mathrm{I}_{2l}(\mathbf{k})=&\sum_{(m,n)\in\mathcal{N}_{2l}}\lambda_{\phi}(m)^2\lambda_{\phi}(n)^2\mathrm{I}(m,n,\mathbf{k}).
\end{split}
\end{equation*}

\begin{lemma}\label{l:gh} Let $\mathbf{n}\in\mathcal{N}_{1l}$ and $\alpha\le y\le\beta$.
  \begin{equation*}
  d(\mathbf{n},y)\gg\frac{y\left|n_1n_2-n_3n_4\right|t_{\phi}^{\frac{3}{2}l\epsilon}}{t_{\phi}},\leqno(1)
  \end{equation*}
  \[
  R(\mathbf{n},\mathbf{k},y)\ll\frac{1}{t_{\phi}^{2-(l+1)\epsilon}}.\leqno(2)
  \]
\end{lemma}

\begin{proof}[Proof of Lemma \ref{l:gh}] As in the proof of Lemma \ref{l:h}, we recall \eqref{e:ch}
\[
\left|h_1+h_2-h_3-h_4\right|\ge\frac{(2\pi y)^2\left|n_1n_2-n_3n_4\right|t_{\phi}^2}{\left(h_1h_2+h_3h_4\right)\left(h_1+h_2+h_3+h_4\right)}.
\]
From this, (1) follows. (2) is easy.
\end{proof}

By Lemma \ref{l:gh} and the fact that $R(\mathbf{n},\mathbf{k},y)$ is increasing in $y$, as in the proof of the upper bound for $\mathrm{I}(\mathbf{n})$, we have
\[
\mathrm{I}(\mathbf{n},\mathbf{k})\ll\frac{t_{\phi}^{-\frac{l}{2}\epsilon}t_{\phi}^{\epsilon}}{t_{\phi}\left|n_1n_2-n_3n_4\right|}
\]
and then
\[
\mathrm{I}_{1l}(\mathbf{k})\ll t_{\phi}^{3\epsilon}.
\]
As in the proof of the upper bound for $\mathrm{I}_0$,
\[
\mathrm{I}_{2l}(\mathbf{k})\ll\sum_{l<t_{\phi}}\sum_{\substack{m<n\\ m<\frac{l}{2}}}\lambda_{\phi}(m)^2\lambda_{\phi}(n)^2\cdot\frac{\log t_{\phi}}{nt_{\phi}}\ll t_{\phi}^{2\epsilon}.
\]
From these, we get
\[
\int_{-1/2}^{1/2}\int_{1/2}^{\infty}\left|\psi_{1l}(z)\right|^4dydx\ll t_{\phi}^{10\epsilon}\qquad(1\le l\le l_{\epsilon}).
\]
Thus, we prove \eqref{e:ub} for $f(z)=\psi_{1l}(z)$ ($1\le l\le l_{\epsilon}$).

For $\psi_{3}(z)$, we recall
\[
\mathrm{H}(\xi)=\sqrt{\xi^2-1}-\mathrm{arcsec}(\xi)\qquad(\xi>1)
\]
and by Lemma \ref{l:kbessel} (2)
\[
\mathrm{K}_{ir}(u)\ll\frac{e^{-t_{\phi}\mathrm{H}\left(\frac{u}{r}\right)}}{\left(u^2-r^2\right)^{\frac{1}{4}}}\qquad(u\ge r+r^{\epsilon+\frac{1}{3}}).
\]
Applying these, we obtain
\[
\int_{-1/2}^{1/2}\int_{1/2}^{\infty}\left|\psi_3(z)\right|^4dydx\ll t_{\phi}^{\epsilon}.
\]
Thus, we prove \eqref{e:ub} for $f(z)=\psi_3(z)$.

We have completed the proof of Theorem \ref{t:l4}.

\section{Proof of Theorem \ref{t:main2}}

Set
\[
\varphi_x(y)=\psi_y(x)=\sum_{n\not=0}\lambda_{\phi}(n)e(nx)e^{\frac{\pi}{2}t_{\phi}}\mathrm{K}_{it_{\phi}}(2\pi|n|y).
\]
Define
\begin{equation*}\begin{split}
J(\varphi_x,\eta)=&\frac{1}{h}\int_a^b\left|\int_0^{\eta}\varphi_x(y+\theta)d\theta\right|dy\qquad(a>0,\,h>1,\,b=a+h),\\
J(\psi_y,\eta)=&\int_{-1/2}^{1/2}\left|\int_0^{\eta}\psi_y(x+\theta)d\theta\right|dx.
\end{split}
\end{equation*}

\begin{lemma}\label{l:J} Let $0<\delta<1/100$ and $0<\delta_1<\delta/10^7$. 
Put
\[
L_x(s)=\sum_{1\le n\le t_{\phi}}\frac{\lambda_{\phi}(n)e(nx)}{n^s},
\]
\begin{equation*}\begin{split}
J_1(x,\delta)=&\frac{\log t_{\phi}}{\sqrt{t_{\phi}}}\int_{t_{\phi}^{1-\delta}}^{t_{\phi}+7\log t_{\phi}}\left|L_x\left(\frac{1}{2}+it\right)\right|^2\frac{dt}{\left(|t_{\phi}-t|+1\right)^{\frac{1}{2}}},\\
J_2(x,\delta)=&\frac{\log t_{\phi}}{t_{\phi}}\int_0^{t_{\phi}^{1-\delta}}\left|L_x\left(\frac{1}{2}+it\right)\right|^2dt.
\end{split}
\end{equation*}
Then, we have
\[
J\left(\varphi_x,\eta\right)^2\ll\frac{1}{t_{\phi}^2}+\frac{J_1(x,\delta)}{t_{\phi}^{2-2\delta}}+\eta^2J_2(x,\delta),
\]
\[
\int_{-\frac{1}{2}}^{\frac{1}{2}}J_1(x,\delta)+t_{\phi}^{\delta}J_2(x,\delta)dx\ll t_{\phi}^{\delta_1}.
\]
\end{lemma}

\begin{proof}[Proof of Lemma \ref{l:J}] By \eqref{e:trun},
\[
J\left(\varphi_x,\eta\right)^2\ll\frac{1}{t_{\phi}^2}+J\left(\widetilde{\varphi}_x,\eta\right)^2,
\]
where
\[
\widetilde{\varphi}_x(y)=\sum_{1\le n\le t_{\phi}}\lambda_{\phi}(n)e(nx)e^{\frac{\pi}{2}t_{\phi}}\mathrm{K}_{it_{\phi}}(2\pi|n|y).
\]
Using the Mellin transform formula of $\mathrm{K}_{it_{\phi}}(2\pi ny)$ \cite[10.32.13]{Ol}, we have
\[
\widetilde{\varphi}_x(y)=\frac{1}{2\pi i}\int_{\left(\frac{1}{2}\right)}L_x(s)\gamma_{\phi}(s)y^{-s}ds,
\]
where
\[
\gamma_{\phi}(s)=e^{\frac{\pi}{2}t_{\phi}}\Gamma\left(\frac{s+it_{\phi}}{2}\right)\Gamma\left(\frac{s-it_{\phi}}{2}\right)\pi^{-s}.
\]
By Stirling's formula,
\[
\gamma_{\phi}\left(\frac{1}{2}+it\right)\ll \begin{cases} e^{-\frac{\pi}{2}\left(|t|-t_{\phi}\right)}\cdot\frac{1}{\left(\left|t_{\phi}-t\right|+1\right)^{\frac{1}{4}}}\cdot\frac{1}{\left(\left|t_{\phi}+t\right|+1\right)^{\frac{1}{4}}}& |t|>t_{\phi}\\ 
\frac{1}{\left(\left|t_{\phi}-t\right|+1\right)^{\frac{1}{4}}}\cdot\frac{1}{\left(\left|t_{\phi}+t\right|+1\right)^{\frac{1}{4}}}& |t|\le t_{\phi}.
\end{cases}
\]
Thus, by this and Proposition \ref{p:approx} (1),
\begin{equation*}\begin{split}
\int_{|t|>t_{\phi}+7\log t_{\phi}}L_x\left(\frac{1}{2}+it\right)\gamma_{\phi}\left(\frac{1}{2}+it\right)y^{-\frac{1}{2}-it}dt
\ll&\int_{t_{\phi}+7\log t_{\phi}}^{\infty}\sum_{1\le n\le t_{\phi}}\frac{\left|\lambda_{\phi}(n)\right|}{\sqrt{n}}\frac{e^{-\frac{\pi}{2}\left(t-t_{\phi}\right)}}{\left(t^2-t_{\phi}^2\right)^{\frac{1}{4}}}dt\\
\ll&\sum_{1\le n\le t_{\phi}}\frac{\left|\lambda_{\phi}(n)\right|}{\sqrt{n}}
\frac{e^{-\frac{7\pi}{2}\log t_{\phi}}}{t_{\phi}^{\frac{1}{4}}}\\
\ll&\frac{1}{t_{\phi}^2}.
\end{split}
\end{equation*}
By this, we have
\[
J\left(\widetilde{\varphi}_x,\eta\right)^2\ll\frac{1}{t_{\phi}^2}+\int_a^b\left|j_{x,1}(y)\right|^2dy+\int_a^b\left|j_{x,1}(y+\eta)\right|^2dy
+\eta\int_0^{\eta}\int_a^b\left|j_{x,2}(y+\theta)\right|^2dyd\theta,
\]
where $h$ is fixed ($b=a+h$),
\begin{equation*}\begin{split}
j_{x,1}(y)=&\int_{t_{\phi}^{1-\delta}}^{t_{\phi}+7\log t_{\phi}}L_x\left(\frac{1}{2}+it\right)\gamma_{\phi}\left(\frac{1}{2}+it\right)\frac{y^{-it}}{\frac{1}{2}-it}dt,\\
j_{x,2}(y)=&\int_0^{t_{\phi}^{1-\delta}}L_x\left(\frac{1}{2}+it\right)\gamma_{\phi}\left(\frac{1}{2}+it\right)y^{-it}dt.
\end{split}
\end{equation*}
By the Cauchy-Schwartz inequality and Stirling's formula, letting
\[
Q(t)=L_x\left(\frac{1}{2}+it\right)\gamma_{\phi}\left(\frac{1}{2}+it\right)\bigg/\left(\frac{1}{2}-it\right),
\]
we have
\begin{equation*}\begin{split}
\int_a^b\left|j_{x,1}(y)\right|^2dy=&\int_{t_{\phi}^{1-\delta}\le t_1,t_2\le t_{\phi}+7\log t_{\phi}}
Q(t_1)Q(-t_2)\frac{b^{1+i(t_2-t_1)}-a^{1+i(t_2-t_1)}}{1+i(t_2-t_1)}dt_2dt_1\\
\ll&\frac{\log t_{\phi}}{t_{\phi}^{2-2\delta}\sqrt{t_{\phi}}}\int_{t_{\phi}^{1-\delta}}^{t_{\phi}+7\log t_{\phi}}\left|L_x\left(\frac{1}{2}+it\right)\right|^2\frac{dt}{\left(|t_{\phi}-t|+1\right)^{\frac{1}{2}}}\\
=&\frac{J_1(x,\delta)}{t_{\phi}^{2-2\delta}}
\end{split}
\end{equation*}
and similarly we obtain
\[
\int_a^b\left|j_{x,1}(y+\eta)\right|^2dy\ll\frac{J_1(x,\delta)}{t_{\phi}^{2-2\delta}},
\]
\[
\int_a^b\left|j_{x,2}(y)\right|^2dy\ll\frac{\log t_{\phi}}{t_{\phi}}\int_0^{t_{\phi}^{1-\delta}}\left|L_x\left(\frac{1}{2}+it\right)\right|^2dt=J_2(x,\delta).
\]
From these, we obtain
\[
J\left(\varphi_x,\eta\right)^2\ll\frac{1}{t_{\phi}^2}+\frac{J_1(x,\delta)}{t_{\phi}^{2-2\delta}}+\eta^2J_2(x,\delta).
\]
Applying Proposition \ref{p:approx}, we have
\begin{equation*}\begin{split}
\int_{-1/2}^{1/2}J_1(x,\delta)dx=&\frac{\log t_{\phi}}{\sqrt{t_{\phi}}}\int_{t_{\phi}^{1-\delta}}^{t_{\phi}+7\log t_{\phi}}\sum_{n<t_{\phi}}\frac{\lambda_{\phi}(n)^2}{n}\frac{1}{\sqrt{\left|t_{\phi}-t\right|+1}}dt\ll t_{\phi}^{\delta_1},\\
\int_{-1/2}^{1/2}J_2(x,\delta)dx=&\frac{\log t_{\phi}}{t_{\phi}}\int_0^{t_{\phi}^{1-\delta}}\sum_{n<t_{\phi}}\frac{\lambda_{\phi}(n)^2}{n}dt\ll t_{\phi}^{-\delta}t_{\phi}^{\delta_1}.
\end{split}
\end{equation*}
Thus, we have
\[
\int_{-1/2}^{1/2}J_1(x,\delta)+t_{\phi}^{\delta}J_2(x,\delta)dx\ll t_{\phi}^{\delta_1}.
\]

We have completed the proof of Lemma \ref{l:J}.
\end{proof}

\begin{lemma}\label{l:Jy} Let $0<\delta<1/2$ and $0<\delta_1<\delta/10^7$. For $y>0$ and $\eta>0$, we have
\[
J(\psi_y,\eta)\ll\frac{t_{\phi}^{\delta_1}\sqrt{\eta}}{\sqrt{t_{\phi}}}+\frac{t_{\phi}^{\delta_1}}{t_{\phi}}
\]
\end{lemma}

\begin{proof}[Proof of Lemma \ref{l:Jy}] We have
\begin{equation*}\begin{split}
J(\psi_y,\eta)^2\ll&\frac{1}{t_{\phi}^2}+\int_{-1/2}^{1/2}\Biggl|\sum_{1\le n\le t_{\phi}}\frac{\lambda_{\phi}(n)}{n}e(nx)\left(e(n\eta)-1\right)e^{\frac{\pi}{2}t_{\phi}}\mathrm{K}_{it_{\phi}}(2\pi ny)\Biggr|^2dx\\
\ll&\sum_{1\le n\le t_{\phi}}\frac{\lambda_{\phi}(n)^2}{n^2}\sin^2(\pi n\eta)e^{\pi t_{\phi}}\mathrm{K}_{it_{\phi}}(2\pi ny)^2.
\end{split}
\end{equation*}
Thus, by Lemma \ref{l:kbessel},
\[
J(\psi_y,\eta)^2\ll J_{1,y}+J_{2,y}+J_{3,y}+J_{4,y},
\]
where
\begin{equation*}\begin{split}
J_{1,y}=&\sum_{n<\frac{t_{\phi}}{4\pi y}}\frac{\lambda_{\phi}(n)^2}{n^2}\frac{1}{\left(t_{\phi}^2-(2\pi ny)^2\right)^{\frac{1}{2}}}\sin^2(\pi n\eta),\\
J_{2,y}=&\frac{1}{t_{\phi}^{2+\frac{1}{2}}}\sum_{\frac{t_{\phi}}{4\pi y}\le n<\frac{t_{\phi}-t_{\phi}^{\frac{1}{3}}}{2\pi y}}\frac{\lambda_{\phi}(n)^2}{\sqrt{t_{\phi}-2\pi ny}},\\
J_{3,y}=&\frac{1}{t_{\phi}^{2}}\sum_{\frac{t_{\phi}-t_{\phi}^{\frac{1}{3}}}{2\pi y}\le n<\frac{t_{\phi}+t_{\phi}^{\frac{1}{3}}}{2\pi y}}
\frac{\lambda_{\phi}(n)^2}{t_{\phi}^{\frac{2}{3}}},\\
J_{4,y}=&\frac{1}{t_{\phi}^{2+\frac{1}{2}}}\sum_{\frac{t_{\phi}-t_{\phi}^{\frac{1}{3}}}{2\pi y}\le n<t_{\phi}}\frac{\lambda_{\phi}(n)^2}{\sqrt{2\pi ny -t_{\phi}}}.
\end{split}
\end{equation*}
Applying Proposition \ref{p:approx} (1), the Cauchy-Schwartz inequality and \eqref{e:ks}, we have
\begin{equation*}\begin{split}
J_{1,y}\ll&\frac{\eta}{t_{\phi}}\sum_{n<\frac{t_{\phi}}{4\pi y}}\frac{\lambda_{\phi}(n)^2}{n}\ll\frac{\eta}{t_{\phi}}t_{\phi}^{\delta_1/3},\\
J_{2,y}^2\ll&\frac{1}{t_{\phi}^{4+1}}\log t_{\phi}\sum_{\frac{t_{\phi}}{4\pi y}\le n\frac{t_{\phi}-t_{\phi}^{\frac{1}{3}}}{2\pi y}}\lambda_{\phi}(n)^4\ll\frac{t_{\phi}^{\delta_1/3}}{t_{\phi}^4},\\
J_{3,y}\ll&\frac{1}{t_{\phi}^{2}}\frac{t_{\phi}^{\frac{1}{3}}t_{\phi}^{\frac{1}{4}}}{t_{\phi}^{\frac{2}{3}}}\ll\frac{1}{t_{\phi}^2},\\
J_{4,y}^2\ll&\frac{t_{\phi}^{\delta_1/3}}{t_{\phi}^4}.
\end{split}
\end{equation*}
From these, Lemma \ref{l:Jy} follows.
\end{proof}

\begin{lemma}\label{l:l12} Let $a>0$, $\epsilon>0$, $0<\epsilon_1<\epsilon/10^7$ and $0<\epsilon_2<\epsilon_1/10^7$.

\noindent $(1)$ For some $Y_k$ in $\mathbf{v}_k$, $1\le k<t_{\phi}^{1-\epsilon}$,
\[
t_{\phi}^{-(\epsilon_1-\epsilon_2)/2}\ll\int_{-\frac{1}{2}}^{\frac{1}{2}}\psi_{Y_k}(x)^2dx\ll t_{\phi}^{\epsilon_1}\left[\int_{-\frac{1}{2}}^{\frac{1}{2}}\left|\psi_{Y_k}(x)\right|dx\right]^2,
\]
except for $\left[t_{\phi}^{1-\epsilon-(\epsilon_1-2\epsilon_2)}\right]$ many $k$'s.

\noindent $(2)$ Let $\delta=c\epsilon$ $(c>1/2)$ and $0<\epsilon_3<\min\left\{\delta/10^7,\, \epsilon_2/10^7\right\}$. We can find $a>0$ and  $h>0$ such that for some $X_k$ in $\mathbf{h}_k$, $1\le k<t_{\phi}^{1-\epsilon}$,
\[
t_{\phi}^{-(\epsilon_1-\epsilon_2)/2}\ll\int_a^{a+h}\varphi_{X_k}(y)^2dy\ll t_{\phi}^{\epsilon_1}\left[\int_a^{a+h}\left|\varphi_{X_k}(y)\right|dy\right]^2,
\]
\[
J_1(X_k,\delta)+t_{\phi}^{\delta}J_1(X_k,\delta)\le t_{\phi}^{\epsilon_1+\epsilon_3}
\]
except for $\left[t_{\phi}^{1-\epsilon-(\epsilon_1-2\epsilon_2)}\right]$ many $k$'s, where $J_1(x,\delta)$ and $J_2(x,\delta)$ are in Lemma \ref{l:J}.
\end{lemma}

\begin{proof}[Proof of Lemma \ref{l:l12}] (1). Let $M>0$.
$\mathcal{S}_M$ denotes the set of all $\mathbf{v}_k$'s such that
\[
t_{\phi}^{\epsilon_2}+\int_{-\frac{1}{2}}^{\frac{1}{2}}\psi_y(x)^4dx\ge M\left[\int_{-\frac{1}{2}}^{\frac{1}{2}}\psi_y(x)^2dx\right]^2
\]
for all $y\in \mathbf{v}_k$.
By Theorem \ref{t:l4},
\begin{equation*}
M\sum_{\mathbf{v}_k\in\mathcal{S}_M}\int_{\mathbf{v}_k}\left[\int_{-\frac{1}{2}}^{\frac{1}{2}}\psi_y(x)^2dx\right]^2dy\le
\sum_{\mathbf{v}_k\in\mathcal{S}_M}\int_{\mathbf{v}_k}\int_{-\frac{1}{2}}^{\frac{1}{2}}\psi_y(x)^4dxdy+t_{\phi}^{\epsilon_2}\ll t_{\phi}^{\epsilon_2}.
\end{equation*}
Thus,
\begin{equation}\label{e:mless1}
M\sum_{\mathbf{v}_k\in\mathcal{S}_M}\int_{\mathbf{v}_k}\left[\int_{-\frac{1}{2}}^{\frac{1}{2}}\psi_y(x)^2dx\right]^2dy\ll t_{\phi}^{\epsilon_2}.
\end{equation}
Recall
\[
\Delta=t_{\phi}^{\frac{1}{3}}\log t_{\phi}\qquad \text{and}\qquad H(n,y)=t_{\phi}\mathrm{H}\left(\frac{2\pi ny}{t_{\phi}}\right).
\]
By the Cauchy-Schwarz inequality and Proposition \ref{p:approx} (1), we have
\begin{multline*}
\sum_{|n|\le\frac{t_{\phi}-\Delta}{2\pi y}}\frac{\left|\lambda_{\phi}(n)\right|^2}
{\left(t_{\phi}^2-(2\pi ny)^2\right)^{\frac{1}{2}}}
\sin^2\left(\frac{\pi}{4}+H(n,y)\right)\\
\le
\Biggl(\sum_{|n|\le\frac{t_{\phi}-\Delta}{2\pi y}}\frac{\left|\lambda_{\phi}(n)\right|^4}{t_{\phi}}\Biggr)^{\frac{1}{2}}
\Biggl(\sum_{|n|\le\frac{t_{\phi}-\Delta}{2\pi y}}\frac{1}{t_{\phi}-2\pi ny}\Biggr)^{\frac{1}{2}}\ll_{\delta}t_{\phi}^{\delta}
\end{multline*}
for a small $\delta>0$.
By this and Proposition \ref{p:approx} (2), we have
\begin{multline}\label{e:lowerbd}
\int_{\mathbf{v}_k}\left[\int_{-\frac{1}{2}}^{\frac{1}{2}}\psi_y(x)^2dx\right]^2dy \\
\gg\sum_{\frac{t_{\phi}}{4\pi Y10^5}\le m,n\le\frac{t_{\phi}}{4\pi Y}}
\frac{\left|\lambda_{\phi}(m)\right|^2\left|\lambda_{\phi}(n)\right|^2}{t_{\phi}^2}S(m,n)
+O\left(\left|\mathbf{v}_k\right|t_{\phi}^{2\theta-\frac{1}{3}+\delta}\log t_{\phi}\right)
\end{multline}
for $\mathbf{v}_k\in\mathcal{S}_M$, where $Y>2$ and
\[
S(m,n)=\int_{\mathbf{v}_k}
\sin^2\left(\frac{\pi}{4}+H(m,y)\right)\sin^2\left(\frac{\pi}{4}+H(n,y)\right)dy
\]
For $\frac{t_{\phi}}{4\pi Y10^5}\le m,n\le\frac{t_{\phi}}{4\pi Y}$, by applying integration by parts, we have
\[
\int_{\mathbf{v}_k}\sin\left(2H(n,y)\right)dy=O\left(\frac{1}{t_{\phi}}\right),
\]
\[
\int_{\mathbf{v}_k}\cos\left(2H(m,y)+2H(n,y)\right)dy=O\left(\frac{1}{t_{\phi}}\right),
\]
because by Lemma \ref{l:kbessel} (4), (5), we have
\begin{equation*}\begin{split}
\frac{\partial}{\partial y}\left(H(n,y)\right)=&-\frac{\sqrt{t_{\phi}^2-(2\pi ny)^2}}{y}\gg t_{\phi},\\
\frac{\partial^2}{\partial y^2}\left(H(n,y)\right)=&\frac{t_{\phi}^2}{y^2\sqrt{t_{\phi}^2-(2\pi ny)^2}}=O\left(t_{\phi}\right)
\end{split}
\end{equation*}
for $\frac{t_{\phi}}{4\pi Y10^5}\le n\le\frac{t_{\phi}}{4\pi Y}$.
By this and the fact that for real values $a,b$, we have
\begin{equation*}\begin{split}
4\sin^2\left(\frac{\pi}{4}+a\right)\sin^2\left(\frac{\pi}{4}+b\right)=&1+\sin\left(2a\right)+\sin\left(2b\right)+\sin\left(2a\right)\sin\left(2b\right)\\
\ge&\frac{1}{2}+\sin\left(2a\right)+\sin\left(2b\right)-\frac{\cos\left(2a+2b\right)}{2},
\end{split}\end{equation*}
we get
\[
S(m,n)\ge\frac{\left|\mathbf{v}_k\right|}{8}+O\left(\frac{1}{t_{\phi}}\right).
\]
By this, \eqref{e:lowerbd},
Proposition \ref{p:rhobounds} (1), \eqref{e:rho} and the fact from \eqref{e:ks} that $2\theta-\frac{1}{3}<0$,  for $\mathbf{v}_k\in\mathcal{S}_M$,  we have
\[
\int_{\mathbf{v}_k}\left[\int_{-\frac{1}{2}}^{\frac{1}{2}}\psi_y(x)^2dx\right]^2dy\gg \left|\mathbf{v}_k\right|
\]
for a sufficiently small $\delta>0$.
By this and \eqref{e:mless1}, we get
\begin{equation}\label{e:supper}
M\left|\mathcal{S}_M\right|\frac{1}{t_{\phi}^{1-\epsilon}}\ll t_{\phi}^{\epsilon_2+\epsilon_3}
\qquad\text{thus}\qquad
\left|\mathcal{S}_M\right|\ll\frac{t_{\phi}^{1-\epsilon+2\epsilon_2}}{M}.
\end{equation}
By the
definition of $\mathcal{S}_M$, there exists $Y_k\in \mathbf{v}_k$ such that
\[
t_{\phi}^{\epsilon_2}+\int_{-\frac{1}{2}}^{\frac{1}{2}}\psi_{Y_k}(x)^4dx\le M\left[\int_{-\frac{1}{2}}^{\frac{1}{2}}\psi_{Y_k}(x)^2dx\right]^2.
\]
By this and the H\"older inequality
\[
\int_{-\frac{1}{2}}^{\frac{1}{2}}\psi_{Y_k}(x)^2dx\le\left[\int_{-\frac{1}{2}}^{\frac{1}{2}}\left|\psi_{Y_k}(x)\right|dx\right]^{\frac{2}{3}}
\left[\int_{-\frac{1}{2}}^{\frac{1}{2}}\psi_{Y_k}(x)^4dx\right]^{\frac{1}{3}},
\]
we have

\begin{equation}\label{e:yklower}
\int_{-\frac{1}{2}}^{\frac{1}{2}}\psi_{Y_k}(x)^2dx\le M\left[\int_{-\frac{1}{2}}^{\frac{1}{2}}\left|\psi_{Y_k}(x)\right|dx\right]^2
\end{equation}
for $\mathbf{v}_k\notin\mathcal{S}_M$. Finally, due to \eqref{e:supper} and \eqref{e:yklower}, by choosing $M=t_{\phi}^{\epsilon_1}$, Lemma \ref{l:l12}] (1) follows.

(2) Similarly, we follow the proof of the previous statement (1).
Let
\[
M>0\qquad\text{and}\qquad h>0.
\]
$\mathcal{T}_M$ denotes the set of all $\mathbf{h}_k$'s such that
\[
t_{\phi}^{\epsilon_2}+\int_a^{a+h}\phi_x(y)^4dy+J(x,\delta)\ge M\left[\int_a^{a+h}\phi_x(y)^2dy\right]^2
\]
for all $x\in \mathbf{h}_k$,
where
\[
J(x,\delta)=J_1(x,\delta)+t_{\phi}^{\delta}J_2(x,\delta).
\]
By Theorem \ref{t:l4} and Lemma \ref{l:J},
\begin{equation*}
M\sum_{\mathbf{h}_k\in\mathcal{T}_M}\int_{\mathbf{h}_k}\left[\int_a^{a+h}\phi_x(y)^2dy\right]^2dx\le t_{\phi}^{\epsilon_2}+
\sum_{\mathbf{h}_k\in\mathcal{T}_M}\int_{\mathbf{h}_k}\left(\int_a^{a+h}\phi_x(y)^4dy+J(x,\delta)\right)dx\ll t_{\phi}^{\epsilon_2}.
\end{equation*}
Thus,
\begin{equation}\label{e:tmless}
M\sum_{\mathbf{h}_k\in\mathcal{T}_M}\int_{\mathbf{h}_k}\left[\int_a^{a+h}\phi_x(y)^2dy\right]^2dx\ll t_{\phi}^{\epsilon_2}.
\end{equation}

\begin{claim} For $\mathbf{h}_k\in\mathcal{T}_M$, we can choose $a>0$ and $h>0$ such that
\[
\int_{\mathbf{h}_k}\left[\int_a^{a+h}\phi(x+iy)^2dy\right]^2dx\gg\left|\mathbf{h}_k\right|=\frac{1}{t_{\phi}^{1-\epsilon}}.
\]
\end{claim}

\begin{proof}[Proof of Claim] This follows from the methods in \cite[pp. 1549--1558]{GRS}. For convenience, we briefly introduce the arguments.
For a large fixed $h>0$, let $a,b>0$ and $k(y)\in C_0^{\infty}(\mathbb{R})$ be an even function such that $(a-b,2a+b)\subset(Y,Y+h)$ and for $y\ge0$,
\begin{enumerate}
\item $k(y)=1$ for $a\le y\le2a$,
\item $k(y)=0$ for $y\in[0,a-b]\cup[2a+b,\infty)$,
\item $0\le k(y)\le1$ for $y\in(a-b,a)\cup(2a,2a+b)$,
\item $\left|k^{(l)}(y)\right|\le C_lb^{-l}$ for some constant $C_l>0$ with any $l\le0$.
\end{enumerate}

We use
\[
\phi(x+iy)=2\sqrt{y}\sum_{n=1}^{\infty}\rho_{\phi}(n)\cos(2\pi nx)\mathrm{K}_{it_{\phi}}(2\pi ny).
\]
Set
\[
N=\frac{t_{\phi}}{\nu}\qquad(\nu>0).
\]
Define
\[
F(y)=2\sqrt{y}\sum_{n\le N}\widetilde{\rho}_{\phi}(n)\cos(2\pi nx)e^{\frac{\pi}{w}t_{\phi}}\mathrm{K}_{it_{\phi}}(2\pi ny)
\]
We have
\[
I:=\int_Y^{Y+h}\phi(x+yi)^2k(y)\frac{dy}{y}\ge 2I_1-I_2,
\]
where
\[
I_1=2\int F(y)\phi(x+iy)k(y)\frac{dy}{y}\qquad{and}\qquad I_2=\int F(y)^2k(y)\frac{dy}{y}.
\]
We write
\[
I_j=4\sum_{m\ge N}\sum_{n\ge1}\widetilde{\rho}_{\phi}(m)\widetilde{\rho}_{\phi}(n)\cos(2\pi mx)\cos(2\pi nx)\alpha_j(n)G(m,n),
\]
where
\[
G(m,n)=e^{\pi t_{\phi}}\int\mathrm{K}_{it_{\phi}}(2\pi my)\mathrm{K}_{it_{\phi}}(2\pi ny)k(y)
\]
and with $\alpha_1(n)=1$ while $\alpha_2(n)$ does the same except it vanishes for $n<N$. From Proposition 6.4 in \cite[p., 1551]{GRS} and Proposition \ref{p:rhobounds}, we have
\[
\mathop{\sum_{n\ge N}\sum_{n\ge1}}_{|m-n|\ge1}\left|\widetilde{\rho}_{\phi}(m)\widetilde{\rho}_{\phi}(n)G(m,n)\right|\ll_l\frac{\sqrt{\nu}}{b^l}+\frac{a^{-\frac{1}{2}}}{b}+\frac{a\sqrt{\nu}}{b^3}+o(1).
\]
Thus,
\begin{multline*}
\int_{\mathbf{h}_k}Idx\ge\int_{\mathbf{h}_k}4\sum_{n\ge N}\rho_{\phi}(n)^2\cos^2(2\pi mx)\int\mathrm{K}_{it_{\phi}}(2\pi ny)^2k(y)dydx\\
+O\left(\left|\mathbf{h}_k\right|\left(\frac{\sqrt{\nu}}{b^l}+\frac{a^{-\frac{1}{2}}}{b}+\frac{a\sqrt{\nu}}{b^3}\right)\right)+o(\left|\mathbf{h}_k\right|).
\end{multline*}
By this and the fact that for $n\ge N$,
\[
\int_{\mathbf{h}_k}\cos^2(2\pi nx)dx=\frac{1}{2}\int_{\mathbf{h}_k}1+\cos(4\pi nx)dx=\frac{1}{2}\left(\left|\mathbf{h}_k\right|+O\left(\frac{\nu}{t_{\phi}}\right)\right),
\]
we have
\begin{multline*}
\int_{\mathbf{h}_k}Idx\ge\left|\mathbf{h}_k\right|\sum_{n\ge N}\rho_{\phi}(n)^2\int\mathrm{K}_{it_{\phi}}(2\pi ny)^2k(y)dy\\
+O\left(\left|\mathbf{h}_k\right|\left(\frac{\sqrt{\nu}}{b^l}+\frac{a^{-\frac{1}{2}}}{b}+\frac{a\sqrt{\nu}}{b^3}\right)\right)+o(\left|\mathbf{h}_k\right|).
\end{multline*}
By this and the fact from \cite[p. 1557]{GRS} that
\[
\sum_{n\ge N}\rho_{\phi}(n)^2\int\mathrm{K}_{it_{\phi}}(2\pi ny)^2k(y)dy\ge2\log2+O\left(a\nu^{-1}\right),
\]
we get
\[
\int_{\mathbf{h}_k}Idx\gg\left|\mathbf{h}_k\right|
\]
by choosing
\[
b=\sqrt{\nu},\,\,\,l=100,\,\,\,\nu=a^{\frac{101}{100}}
\]
with a sufficiently large $a$. Thus, for $h=b-a$, we have
\[
\left|\mathbf{h}_k\right|\ll\int_{\mathbf{h}_k}\int_a^{a+h}\phi(x+yi)^2dydx.
\]
By the Cauchy-Schwarz inequality, we get
\[
\left|\mathbf{h}_k\right|\ll\int_{\mathbf{h}_k}\int_a^{a+h}\phi(x+yi)^2dydx\le\left(\int_{\mathbf{h}_k}\left[\int_a^{a+h}\phi(x+yi)^2dy\right]^2dx\right)^{\frac{1}{2}}
\left(\int_{\mathbf{h}_k}1dx\right)^{\frac{1}{2}}
\]
or
\[
\left|\mathbf{h}_k\right|\ll\int_{\mathbf{h}_k}\left[\int_a^{a+h}\phi(x+yi)^2dy\right]^2dx
\]
This proves Claim.
\end{proof}

By Claim, \eqref{e:tmless}, \eqref{e:rho} and \eqref{e:lambda} , we have
\[
Mt_{\phi}^{-\epsilon_2}\left|\mathcal{T}_M\right|\cdot\frac{1}{t_{\phi}^{1-\epsilon}}\le t_{\phi}^{\epsilon_2}\qquad\text{or}\qquad\left|\mathcal{T}_M\right|\ll \frac{t_{\phi}^{1-\epsilon+2\epsilon_2}}{M}.
\]
We choose $M=t_{\phi}^{\epsilon_1}$.
Then, for $\mathbf{h}_k\not\in \mathcal{T}_M$, there exists $X_k\in\mathbf{h}_k$ such that
\[
t_{\phi}^{\epsilon_2}+\int_a^{a+h}\phi_x(y)^4dy+J(x,\delta)\le t_{\phi}^{\epsilon_1}\left[\int_a^{a+h}\phi_x(y)^2dy\right]^2
\]
Using this, as in the proof of (1), we similarly obtain the first statement of Lemma \ref{l:l12} (2). For the second statement, we use
\[
J(x,\delta)\le t_{\phi}^{\epsilon_1}\left[\int_a^{a+h}\phi_x(y)^2dy\right]^2.
\]
By this and the fact that due to Theorem \ref{t:GRS},
\[
\left[\int_a^{a+h}\phi_x(y)^2dy\right]^2\ll t_{\phi}^{\epsilon_3},
\]
the second statement follows.

W have completed the proof of Lemma \ref{l:l12}.
\end{proof}

We state the following main lemma.

\begin{lemma}\label{l:main} Let $a>0$. Let $0<\epsilon<1/2$ and $0<\epsilon_1<\epsilon/10^7$.

\noindent $(1)$
We set
\[
\mathcal{A}=\,\,\,\text{the set of all $Y_k$'s that satisfy Lemma \ref{l:l12} (1) for $a$ and $\epsilon$.}
\]
Then, for $\beta=\{x+Y_ki:-\frac{1}{2}<x\le\frac{1}{2}\}$ $\left(Y_k\in\mathcal{A}\right)$ , we have
\[
\textrm{K}^{\beta}(\phi)\gg_{\epsilon}t_{\phi}^{1-5\epsilon_1}.
\]
\noindent $(2)$ We set
\[
\mathcal{B}=\,\,\,\text{the set of all $X_k$'s that satisfy Lemma \ref{l:l12} (2) for $a$, $\epsilon$ and $h$.}
\]
Then,
for $\beta=\{X_k+yi:a\le y\le a+h\}$ $\left(X_k\in\mathcal{B}\right)$ , we have
\[
\textrm{K}^{\beta}(\phi)\gg_{\epsilon}t_{\phi}^{1-9\epsilon_1}.
\]
\end{lemma}

For the proof of this main lemma, we adopt the methods in \cite{Lit}.

\begin{proof}[Proof of Lemma \ref{l:main}] (1) Let's apply Theorem \ref{t:lit}. Define
\[
M_1\left(\psi_{Y_k}\right)=\int_{-1/2}^{1/2}\left|\psi_{Y_k}(x)\right|dx,\qquad M_2\left(\psi_{Y_k}\right)=\Biggl(\int_{-1/2}^{1/2}\left|\psi_{Y_k}(x)\right|^2dx\Biggr)^{1/2}.
\]
By Lemma \ref{l:l12} (1) and Lemma \ref{l:Jy}, for $0<\epsilon_2<\epsilon_1/10^7$ and $0<\delta_1<\epsilon_2/10^7$, we have
\[
M_1\left(\psi_{Y_k}\right)\gg t_{\phi}^{-\epsilon_1/2}M_2\left(\psi_{Y_k}\right),\,\,\,M_2\left(\psi_{Y_k}\right)\gg t_{\phi}^{-\frac{\epsilon_1-\epsilon_2}{4}},\,\,\, J\left(\psi_{Y_k},\eta\right)\ll\frac{t_{\phi}^{\delta_1}\sqrt{\eta}}{\sqrt{t_{\phi}}}+\frac{t_{\phi}^{\delta_1}}{t_{\phi}}
\]
for $\eta=t_{\phi}^{4\epsilon_1}/t_{\phi}$.
Set
\[
N=t_{\phi},\qquad\omega=t_{\phi}^{4\epsilon_1}.
\]
Then, for $\eta=t_{\phi}^{4\epsilon_1}/t_{\phi}$, we have
\[
J\left(\psi_{Y_k},\eta\right)=o\left(\left(t_{\phi}^{-\frac{\epsilon_1}{2}}\right)^3t_{\phi}^{-\frac{\epsilon_1-\epsilon_2}{4}}\eta\right)=
o\left(\left(t_{\phi}^{-\frac{\epsilon_1}{2}}\right)^3M_2\left(\psi_{Y_k}\right)\eta\right).
\]
Thus, by Theorem \ref{t:lit} with $c\asymp t_{\phi}^{-\frac{\epsilon_1}{2}}$, we have
\[
K^{\beta}(\phi)\ge\frac{c^2}{10(\omega+2)}N\gg\left(t_{\phi}^{-\frac{\epsilon_1}{2}}\right)^2\cdot\frac{t_{\phi}}{t_{\phi}^{4\epsilon_1}}=t_{\phi}^{1-5\epsilon_1},
\]
where $\beta=\{x+iY_k:-1/2\le x<1/2\}$.

(2) Similarly, we apply Theorem \ref{t:lit}. Define
\[
M_1\left(\varphi_{X_k}\right)=\int_a^{a+h}\left|\psi_{X_k}(y)\right|dy,\qquad M_2\left(\varphi_{X_k}\right)=\Biggl(\int_a^{a+h}\left|\varphi_{X_k}(y)\right|^2dy\Biggr)^{1/2}.
\]
By Lemma \ref{l:l12} (2) and Lemma \ref{l:J},
\[
M_1\left(\varphi_{X_k}\right)\gg t_{\phi}^{-\frac{\epsilon_1}{2}}M_2\left(\varphi_{X_k}\right),\qquad M_2\left(\varphi_{X_k}\right)\gg t_{\phi}^{-\frac{\epsilon_1-\epsilon_2}{4}}
\]
\[
J\left(\varphi_{X_k},\eta\right)\ll\frac{1}{t_{\phi}}+\frac{\sqrt{J_1\left(X_k,\delta\right)}}{t_{\phi}^{1-\delta}}+\eta\sqrt{J_2\left(X_k,\delta\right)}\ll
\frac{t_{\phi}^{(\epsilon_1+\epsilon_3)/2}}{t_{\phi}^{1-\delta}}+\eta t_{\phi}^{-\frac{\delta}{2}+\frac{\epsilon_1+\epsilon_3}{2}}
\]
for $\eta=t_{\phi}^{8\epsilon_1}h/t_{\phi}$.
Set
\[
N=t_{\phi},\qquad\omega=t_{\phi}^{8\epsilon_1}.
\]
Then, by choosing $\delta=5\epsilon_1$ with $\eta=t_{\phi}^{8\epsilon_1}h/t_{\phi}$, we get
\[
J\left(\varphi_{X_k},\eta\right)=o\left(\left(t_{\phi}^{-\frac{\epsilon_1}{2}}\right)^3M_2\left(\varphi_{X_k}\right)\eta\right).
\]
Thus, by Theorem \ref{t:lit} with $c\asymp t_{\phi}^{-\frac{\epsilon_1}{2}}$, we have
\[
K^{\beta}(\phi)\ge\frac{c^2}{10(\omega+2)}N\gg\left(t_{\phi}^{-\frac{\epsilon_1}{2}}\right)^2\cdot\frac{t_{\phi}}{t_{\phi}^{8\epsilon_1}}=t_{\phi}^{1-9\epsilon_1},
\]
where $\beta=\{X_k+iy:a\le y<a+h\}$.

We have completed the proof of Lemma \ref{l:main}

\end{proof}

Theorem \ref{t:main2} follows from Lemma \ref{l:main} and Lemma \ref{l:l12}.

\section{Proof of Theorem \ref{t:sc}}

By Theorem \ref{t:GRS} with \eqref{e:rho}, we have
$\beta=\{yi:a<y<a+h\}$ ($h>0$) satisfying
\begin{equation}\label{e:l2lowb}
\int_{\beta}\left|\varphi_0(y)\right|^2dy\gg_{\delta}t_{\phi}^{-\delta}
\end{equation}
for any $\delta>0$.
By the Lindel\"of hypothesis for $L(s,\phi)$, we have
\[
L_0\left(\frac{1}{2}+it\right)\ll_{\epsilon}t_{\phi}^{\epsilon}
\]
for $|t|\le t_{\phi}^{O(1)}$ and for any $\epsilon>0$.  From this, we obtain
\[
J_1(0,\delta)+t_{\phi}^{\delta}J_2(0,\delta)\ll t_{\phi}^{\delta_1}
\]
for $0<\delta_1<\delta/10^7$ ($0<\delta<1/100$).
By this and Lemma \ref{l:J}, we have
\[
J\left(\varphi_0,\eta\right)\ll\frac{t_{\phi}^{\delta_1/2}}{t_{\phi}^{1-\delta}}+\eta t_{\phi}^{-\frac{\delta}{2}+\frac{\delta_1}{2}}
\]
for $\eta=t_{\phi}^{11\delta/7}h/t_{\phi}$.
Applying the assumption
\[
\int_{\beta}\phi(iy)^4dy\ll_{\epsilon}t_{\phi}^{\epsilon}
\]
for any $\epsilon>0$, by \eqref{e:l2lowb} and the H\"older inequality
\[
\int_{\beta}|\varphi_0|^2dy\le\left(\int_{\beta}|\varphi_0|dy\right)^{\frac{2}{3}}\left(\int_{\beta}|\varphi_0|^4dy\right)^{\frac{1}{3}},
\]
we have
\[
t_{\phi}^{-\delta_1}M_2\left(\varphi_0\right)\ll_{\delta_1} M_1\left(\varphi_0\right).
\]
Also, we have
\[
M_2\left(\varphi_0\right)\gg_{\delta_1}t_{\phi}^{-\delta_1}.
\]
Let $\epsilon_1>0$.
Set
\[
N=t_{\phi},\qquad\omega=t_{\phi}^{11\epsilon_1}.
\]
Thus, by choosing $\delta=8\epsilon_1$ with $\eta=t_{\phi}^{11\epsilon_1}h/t_{\phi}$
\[
J\left(\varphi_0,\eta\right)=o\left(t_{\phi}^{-3\delta_1}M_2\left(\varphi_0\right)\eta\right).
\]
Thus, by Theorem \ref{t:lit} with $c\asymp t_{\phi}^{-\delta_1}$ and $b=a+h$, we have
\[
K^{\beta}(\phi)\gg_{\epsilon_1}t_{\phi}^{1-12\epsilon_1}.
\]
Thus, Theorem \ref{t:sc} follows.

\section{Acknowledgment} I am deeply grateful to the anonymous referees for their many valuable comments and suggestions. I would also like to sincerely thank Professor Peter Sarnak for his encouragement regarding this paper and for generously sharing his insights during my visit in February 2024. I am also thankful to Dr. Donghoon Park for initiating and supporting this project.

\end{document}